\documentclass[twoside,11pt,reqno]{amsart}
\usepackage{amsmath,amssymb,amscd,mathrsfs,epic,wasysym,latexsym,tikz,mathrsfs,cite,hyperref}
\usepackage{pb-diagram}
\usepackage[enableskew]{youngtab}
\usepackage{young}

\makeatletter

\numberwithin{equation}{section}

\newtheorem{Proposition}[equation]{Proposition}
\newtheorem{Lemma}[equation]{Lemma}
\newtheorem{Theorem}[equation]{Theorem}
\newtheorem{Corollary}[equation]{Corollary}
\theoremstyle{definition}  
\newtheorem{Definition}[equation]{Definition}
\newtheorem{Remark}[equation]{Remark}

\newtheorem{Example}[equation]{Example}

\let\<\langle
\let\>\rangle

\newcommand\Comment[2][\relax]{\space\par\medskip\noindent%
   \fbox{\begin{minipage}{\textwidth}\textbf{Comment\ifx\relax#1\else---#1\fi}\newline%
        #2\end{minipage}}\medskip
}

\def\bi{\text{\boldmath$i$}}
\def\bj{\text{\boldmath$j$}}

\def\b1{\text{\boldmath$1$}}

\def\Fil{{\tt Fil}}
\def\MP{\operatorname{MP}}
\def\minp{{\tt mp}}

\newcommand{\tzero}{{\tiny{0}}}
\def\Shapes{{\mathscr S}}
\def\linebr{{\tt sp}}
\newcommand{\Kos}{{\bf K_*}}
\newcommand{\Resol}{{\bf P_*}}
\newcommand{\BasicRes}{{\bf P_*}}
\newcommand{\ResolPrime}{{\bf P_*'}}

\def\gldim{{\operatorname{gd}\,}}
\def\prdim{{\operatorname{pd}\,}}

\def\pmod#1{\text{ }(\text{\rm mod } #1)\,}

\newcommand{\Hom}{\operatorname{Hom}}
\newcommand{\Ext}{\operatorname{Ext}}
\newcommand{\EXT}{\operatorname{EXT}}
\newcommand{\End}{\operatorname{End}}

\newcommand{\im}{\operatorname{im}}
\newcommand{\id}{\operatorname{id}}

\newcommand{\res}{\operatorname{cont}}
\newcommand{\weight}{\operatorname{wt}}

\newcommand{\soc}{\operatorname{soc}}
\newcommand{\head}{\operatorname{head}}

\newcommand{\cha}{\operatorname{char}}
\newcommand{\St}{\operatorname{St}}

\newcommand{\Stab}{\operatorname{Stab}}

\newcommand{\Z}{\mathbb{Z}}

\def\eps{{\varepsilon}}
\def\phi{{\varphi}}
\def\bsi{\text{\boldmath$\sigma$}}
\def\beps{\text{\boldmath$\varepsilon$}}
\def\btau{\text{\boldmath$\tau$}}
\newcommand{\funF}{{\mathcal F}}

\newcommand{\ga}{\gamma}

\newcommand{\la}{\lambda}
\newcommand{\La}{\Lambda}
\newcommand{\al}{\alpha}
\newcommand{\be}{\beta}

\def\Si{\mathfrak{S}}
\newcommand{\si}{\sigma}

\newcommand{\de}{\delta}
\newcommand{\De}{\Delta}
\newcommand{\ka}{\kappa}

\def\id{\mathop{\mathrm {id}}\nolimits}

\newcommand{\Ind}{{\mathrm {Ind}}}

\newcommand{\rad}{{\mathrm {rad}}}

\newcommand{\Res}{{\mathrm {Res}}}

\newcommand{\Q}{{\mathbb Q}}

\newcommand{\A}{{\mathscr A}}
\newcommand{\Laurent}{{\mathscr A}}

\renewcommand{\mod}{\bmod \,}

\def\h{{\mathfrak h}}
\def\g{{\mathfrak g}}
\def\n{{\mathfrak n}}

\def\dx{x}

\def\dtau{\tau}
\def\NH{N\!\!\:H}

\def\X{M}
\def\bed{\text{\boldmath$i$}}

\def\Par{{\mathscr P}}

\def\umu{{\underline{\mu}}}
\def\unu{{\underline{\nu}}}

\def\f{{\mathbf{f}}}
\def\b{\mathfrak{b}}
\def\k{F}

\def\W{\langle I \rangle}
\def\y{y}

\def\T{{\mathtt T}}

\def\Stab{{\mathtt S}}

\def\spa{\operatorname{span}}
\def\height{{\operatorname{ht}}}

\def\wt{{\operatorname{wt}}}

\def\op{{\mathrm{op}}}
\def\re{{\mathrm{re}}}
\def\im{{\mathrm{im}}}

\def\into{{\hookrightarrow}}
\def\Mod#1{#1\!\operatorname{-Mod}}
\def\mod#1{#1\!\operatorname{-mod}}
\def\proj#1{#1\!\operatorname{-proj}}

\renewcommand\O{\mathcal O}

\def\iso{\stackrel{\sim}{\longrightarrow}}

\def\lan{\langle}
\def\ran{\rangle}
\def\HOM{\operatorname{HOM}}
\def\END{\operatorname{END}}
\def\Stand{\bar\Delta}

\def\CH{{\operatorname{ch}_q\:}}
\def\DIM{{\operatorname{dim}_q\:}}

\def\words{{\langle I\rangle}}
\def\shift{{\tt sh}}

\def\Car{{\tt C}}
\def\cc{{\tt c}}

\def\barinv{\mathtt{b}}

{\catcode`\|=\active
  \gdef\set#1{\mathinner{\lbrace\,{\mathcode`\|"8000%
  \let|\midvert #1}\,\rbrace}}
}
\def\midvert{\egroup\mid\bgroup}

\colorlet{darkgreen}{green!50!black}
\tikzset{dots/.style={very thick,loosely dotted},
         greendot/.style={fill,circle,color=darkgreen,inner sep=1.5pt,outer sep=0}
}
\def\greendot(#1,#2){\node[greendot] at(#1,#2){}}

\newenvironment{braid}{
  \begin{tikzpicture}[baseline=6mm,blue,line width=1pt, scale=0.4,
                      draw/.append style={rounded corners},
                      every node/.append style={font=\fontsize{5}{5}\selectfont}]%
  }{\end{tikzpicture}
}

\def\Grid(#1,#2){
  \draw[very thin,gray,step=2mm] (0,0)grid(#1,#2);
  \draw[very thin,darkgreen,step=10mm] (0,0)grid(#1,#2);
}

\newcommand\Tableau[2][\relax]{
  \begin{tikzpicture}[scale=0.5,draw/.append style={thick,black}]
    \ifx\relax#1\relax%
    \else 
      \foreach\box in {#1} { \filldraw[blue!30]\box+(-.5,-.5)rectangle++(.5,.5); }
    \fi
    \newcount\row\newcount\col
    \row=0
    \foreach \Row in {#2} {
       \col=1
       \foreach\k in \Row {
          \draw(\the\col,\the\row)+(-.5,-.5)rectangle++(.5,.5);
          \draw(\the\col,\the\row)node{\k};
          \global\advance\col by 1
       }
       \global\advance\row by -1
    }
  \end{tikzpicture}
}

\newcommand\YoungDiagram[2][\relax]{
  \begin{tikzpicture}[scale=0.5,draw/.append style={thick,black}]
    \ifx\relax#1\relax%
    \else 
    \foreach\box in {#1} {
      \filldraw[blue!30]\box rectangle ++(1,1);
    }
    \fi
    \newcount\row
    \row=0
    \foreach \col in {#2} {
       \draw(1,\the\row)grid ++(\col,1);
       \global\advance\row by -1
    }
  \end{tikzpicture}
}

\begin{document}

\title[KLR algebras]{{\bf Representation theory and cohomology of Khovanov-Lauda-Rouquier algebras}}

\author{\sc Alexander S. Kleshchev}
\address{Department of Mathematics\\ University of Oregon\\
Eugene\\ OR~97403, USA}
\email{klesh@uoregon.edu}


\subjclass[2000]{20C08, 20C30, 05E10}

\thanks{Research supported by the NSF grant no. DMS-1161094 and the Humboldt Foundation.}

\begin{abstract}
This expository paper is based on the lectures given at the program `Modular Representation Theory of Finite and $p$-adic Groups' at the National University of Singapore. We are concerned with recent results on representation theory and cohomology of KLR algebras, with  emphasis on standard module theory. 
\end{abstract}

\maketitle

\section{Set up and motivation}

This expository paper is based on the lectures given at the program `Modular Representation Theory of Finite and $p$-adic Groups' at the National University of Singapore. We are concerned with recent results on representation theory and cohomology of KLR algebra, with  emphasis on standard module theory, as developed in \cite{KR2}, \cite{Kato}, \cite{McN},  \cite{Kcusp}, \cite{BKM}. Some proofs are given, but often we just review  or illustrate the results.  Other topics in the theory of KLR algebras are nicely reviewed in \cite{BrExp}.

\subsection{KLR algebras}
In this paper we will be mainly concerned with KLR algebras of finite 
Lie type. So let $\Car=(\cc_{ij})_{i,j\in I}$ be a {\em Cartan matrix} of {\em finite type}. 
As in \cite[\S 1.1]{Kac}, let $(\h,\Pi,\Pi^\vee)$ be a realization of the Cartan matrix $\Car$, so we have simple roots $\{\al_i\mid i\in I\}$, simple coroots $\{\al_i^\vee\mid i\in I\}$, and a bilinear form $(\cdot,\cdot)$ on $\h^*$ such that $\cc_{ij}=2(\al_i,\al_j)/(\al_i,\al_i)$ for all $i,j\in I$. We normalize $(\cdot,\cdot)$ so that $(\be,\be)=2$ if $\be$ is a short  root. The fundamental dominant weights $\{\La_i\mid i\in I\}$ have the property that $\lan\La_i,\al_j^\vee\ran=\de_{i,j}$, where $\lan\cdot,\cdot\ran$ is the natural pairing between $\h^*$ and $\h$. We have the set of  dominant weights $P_+=\sum_{i\in I}\Z_{\geq 0}\cdot\La_i$ and the positive part of the root lattice 
$Q_+ := \bigoplus_{i \in I} \Z_{\geq 0} \al_i$. For $\alpha \in Q_+$, we write $\height(\alpha)$ for the sum of its 
coefficients when expanded in terms of the $\al_i$'s. 

Denote $\A:=\Z[q,q^{-1}]$ for an indeterminate $q$. For $n\in\Z_{\geq 0}$, define
$$[n]_q:=\frac{q^n-q^{-n}}{q-q^{-1}},\  [n]_q^!:=\prod_{m=1}^n[m]_q.
$$
Given in addition $\alpha \in Q_+$ and a simple root $\al_i$, we let $d_\alpha := (\alpha,
\alpha)/ 2$ and set  
\begin{align*}
q_\al:=q^{d_\al},\  [n]_\al:= [n]_{q_\al},\  [n]_\al^!:=[n]_{q_\al}^!,
q_i:=q_{\al_i}\ [n]_i:=[n]_{\al_i},\ [n]_i^!:=[n]_{\al_i}^!.
\end{align*}


Sequences of elements of $I$ will be called {\em words}. The set of all words is denoted $\words$. If $\bi=i_1\dots i_d$ is a word, we denote $|\bi|:=\al_{i_1}+\dots+\al_{i_d}\in Q_+$. We refer to $|\bi|$ as the {\em content} of the word $\bi$. For any $\al\in Q_+$ we denote 
$$\words_\al:=\{\bi\in\words \mid |\bi|=\al\}.$$ 
If $\al$ is of height $d$, then the symmetric group $\Si_d$ with simple permutations $s_1,\dots,s_{d-1}$ acts transitively on $\words_\alpha$ from the left by place permutations.

Let $F$ be an arbitrary  
field. Define the polynomials 
$\{Q_{ij}(u,v)\in F[u,v]\mid i,j\in I\}$ in the variables $u,v$   
as follows. 
Choose signs $\eps_{ij}$ for all $i,j \in I$ with $\cc_{ij}
< 0$  so that $\eps_{ij}\eps_{ji} = -1$.
Then set: 
\begin{equation}\label{EArun}
Q_{ij}(u,v):=
\left\{
\begin{array}{ll}
0 &\hbox{if $i=j$;}\\
1 &\hbox{if $\cc_{ij}=0$;}\\
\eps_{ij}(u^{-\cc_{ij}}-v^{-\cc_{ji}}) &\hbox{if $\cc_{ij}<0$.}
\end{array}
\right.
\end{equation}
Fix 
$\al\in Q_+$ of height $d$. 
The {\em KLR-algebra} $R_\al$ is an associative graded unital $F$-algebra, given by the generators
\begin{equation}\label{EKLGens}
\{1_\bi\mid \bi\in \words_\al\}\cup\{y_1,\dots,y_{d}\}\cup\{\psi_1, \dots,\psi_{d-1}\}
\end{equation}
and the following relations for all $\bi,\bj\in \words_\al$ and all admissible $r,t$:
\begin{equation}
1_\bi  1_\bj = \de_{\bi,\bj} 1_\bi ,
\quad{\textstyle\sum_{\bi \in \words_\alpha}} 1_\bi  = 1;\label{R1}
\end{equation}
\begin{equation}\label{R2PsiY}
y_r 1_\bi  = 1_\bi  y_r;\quad y_r y_t = y_t y_r;
\end{equation}
\begin{equation}
\psi_r 1_\bi  = 1_{s_r\bi} \psi_r;\label{R2PsiE}
\end{equation}
\begin{equation}
(y_t\psi_r-\psi_r y_{s_r(t)})1_\bi  
= \de_{i_r,i_{r+1}}(\de_{t,r+1}-\de_{t,r})1_\bi;
\label{R6}
\end{equation}
\begin{equation}
\psi_r^21_\bi  = Q_{i_r,i_{r+1}}(y_r,y_{r+1})1_\bi 
 \label{R4}
\end{equation}
\begin{equation} 
\psi_r \psi_t = \psi_t \psi_r\qquad (|r-t|>1);\label{R3Psi}
\end{equation}
\begin{equation}
\begin{split}
&(\psi_{r+1}\psi_{r} \psi_{r+1}-\psi_{r} \psi_{r+1} \psi_{r}) 1_\bi  
\\=
&
\de_{i_r,i_{r+2}}\frac{Q_{i_r,i_{r+1}}(y_{r+2},y_{r+1})-Q_{i_r,i_{r+1}}(y_r,y_{r+1})}{y_{r+2}-y_r}1_\bi.
\end{split}
\label{R7}
\end{equation}
The {\em grading} on $R_\al$ is defined by setting:
$$
\deg(1_\bi )=0,\quad \deg(y_r1_\bi )=(\al_{i_r},\al_{i_r}),\quad\deg(\psi_r 1_\bi )=-(\al_{i_r},\al_{i_{r+1}}).
$$

\vspace{2 mm}

These algebras were defined in \cite{KL1,KL2,R}. It is pointed out in \cite{KL2} and \cite[\S3.2.4]{R} that up to isomorphism the graded $F$-algebra $R_\al$ depends only on the Cartan matrix and $\al$. 

Fix in addition a dominant weight $\La\in P_+$. The corresponding {\em cyclotomic KLR algebra} $R_\al^\La$ is the quotient of $R_\al$ by the following ideal:
\begin{equation}\label{ECyclot}
J_\al^\La:=(y_1^{\lan\La,\al_{i_1}^\vee\ran}1_\bi \mid \bi= i_1\dots i_d\in\words_\al). 
\end{equation}

For a graded algebra $R$, denote by $\Mod{R}$ the abelian category of all graded left $R$-modules,
denoting (degree-preserving) homomorphisms in this category
by $\hom_R$. We write $\cong$ for the isomorphism in this category, and $\simeq$ for the isomorphism in the category of usual modules. 
Let $\mod{R}$ denote
the abelian subcategory of all
finite dimensional graded left $R$-modules and
 $\proj{R}$ denote the additive subcategory  of 
all finitely generated projective graded left $R$-modules. 

We also consider the Grothendieck groups 
 $[\mod{R}]$ and $[\proj{R}]$.
We view $[\mod{R}]$ and $[\proj{R}]$ as $\Laurent$-modules via 
$q^m[M]:=[q^m M]$,
where $q^m M$ denotes the module obtained by 
shifting the grading up by $m$:
$
(q^mM)_n=M_{n-m}.
$
More generally, given a formal Laurent series $f(q) = \sum_{n \in \Z} f_n
q^n$ with coefficients $f_n \in \Z_{\geq 0}$,
$f(q) V$ denotes $\bigoplus_{n \in \Z} q^n V^{\oplus f_n}$.
If $V$ is a locally finite dimensional graded vector space (i.e. the dimension of each graded component $V_n$ is finite),
its {\em graded dimension} is $
\DIM V := \sum_{n \in \Z} (\dim V_n) q^n.
$
Given $M, L \in \mod{R}$ with $L$ irreducible, 
we write $[M:L]_q$ for the corresponding {\em  graded composition multiplicity},
i.e. 
$
[M:L]_q := \sum_{n \in \Z} a_n q^n,
$ 
where $a_n$ is the multiplicity
of $q^n L$ in a graded composition series of $M$.


Given $\alpha, \beta \in Q_+$, we set $
R_{\alpha,\beta} := R_\alpha \otimes 
R_\beta$.  
There is an injective non-unital algebra homomorphism 
$R_{\alpha,\beta}\,\into\, R_{\alpha+\beta},\ 1_\bi \otimes 1_\bj\mapsto 1_{\bi\bj}$,
where $\bi\bj$ is the concatenation of $\bi$ and $\bj$. The image of the identity
element of $R_{\alpha,\beta}$ under this map is
$
1_{\alpha,\beta}:= \sum_{\bi \in \words_\alpha,\,\bj \in \words_\beta} 1_{\bi\bj}.
$ 
We consider the induction and restriction functors: 
\begin{align*}
\Ind_{\alpha,\beta} &:= R_{\alpha+\beta} 1_{\alpha,\beta}
\otimes_{R_{\alpha,\beta}} ?:\Mod{R_{\alpha,\beta}} \rightarrow \Mod{R_{\alpha+\beta}},\\
\Res_{\alpha,\beta} &:= 1_{\alpha,\beta} R_{\alpha+\beta}
\otimes_{R_{\alpha+\beta}} ?:\Mod{R_{\alpha+\beta}}\rightarrow \Mod{R_{\alpha,\beta}},
\end{align*}
which preserve the categories of finite dimensional and finitely generated projective modules.
For $M\in\Mod{R_{\al}},N\in\Mod{R_\be}$, denote  
$
M\circ N:=\Ind_{\al,\be}
M \boxtimes N. 
$ 
Then $[\proj{R}]:=\bigoplus_{\al\in Q_+}\proj{R_\al}$ and $[\mod{R}]:=\bigoplus_{\al\in Q_+}\mod{R_\al}$ are $Q_+$-graded $\A$-algebras,
with multiplication coming from the
induction product $\circ$. 

\begin{Example} \label{ExNH}
{\rm 
For $m \geq 1$ and $i \in I$,
the KLR algebra $R_{m\alpha_i}$ is the
{\em nil-Hecke algebra} $\NH_m$, which is given by generators
$y_1,\dots,y_m$ and 
$\psi_1,\dots,\psi_{m-1}$ and relations:
$y_i y_j = y_j y_i$, $\psi_i y_j = y_j \psi_i$ for $j \neq
i,i+1$, $\psi_i y_{i+1} = y_i \psi_i + 1$, $y_{i+1} \psi_i
= \psi_i y_i + 1$, $\psi_i^2 = 0$,
together with the usual type ${\tt A_m}$ braid relations for
$\psi_1,\dots,\psi_{m-1}$.
It is well known
that the nil-Hecke algebra is a matrix algebra over its center;
see e.g. \cite[\S2]{R2} or \cite[\S4]{KLM} for recent expositions.
Moreover,
writing $w_0$ for the longest element of $\Si_m$,
the degree zero 
element
\begin{equation}\label{idemp}
e_m := \y_2 \y_3^2 \cdots \y_m^{m-1} \psi_{w_{0}}
\end{equation} 
is a primitive idempotent, hence
$P(\alpha_i^m) := q_i^{m(m-1)/2}R_{m\alpha_i} e_m$ is an indecomposable projective $R_{m\alpha_i}$-module.
The degree shift has been chosen so that 
irreducible head $L(\alpha_i^m)$ of $P(\alpha_i^m)$ has graded dimension $[m]_i^!$.
Thus 
$R_{m\alpha_i}\cong [m]^!_i \,P(\alpha_i^m)$ as a left module.
}
\end{Example}

\subsection{Some motivation}\label{SMotiv}
The first reason why representation theory of KLR algebras is interesting is that it can be used to categorify quantum groups. One way to make this statement more precise is as follows. 
Let $\f$ be the quantized enveloping algebra over the field $\Q(q)$   
associated to $\Car$ with standard generators $\{\theta_i\mid i\in I\}$, cf. \cite{Lubook}. 
It is naturally $Q_+$-graded: 
$
\f = \bigoplus_{\alpha
  \in Q_+} \f_\alpha.
$ 
Khovanov and Lauda showed that there is a
unique $Q_+$-graded algebra isomorphism
\begin{equation}\label{EGamma}
\gamma: \f \stackrel{\sim}{\rightarrow}
\Q(q)
\otimes_{\A} [\proj{R}],
\ 
\theta_i \mapsto [R_{\alpha_i}],
\end{equation}
where $R_{\alpha_i}$ is the left regular module over the algebra $R_{\al_i}$. 
If $\Car$ is symmetric and $F$ has characteristic zero, 
Rouquier \cite{R2}  and Varagnolo and Vasserot \cite{VV} have
shown further that $\gamma$ maps 
the canonical basis of $\f$ to
the basis for $[\proj{R}]$ arising from the 
isomorphism classes of graded self-dual 
indecomposable projective modules. Taking a dual map to $\gamma$ yields another algebra isomorphism
\begin{equation}\label{EGamma*}
\gamma^*:\Q(q) \otimes_{\A} [\mod{R}] 
\stackrel{\sim}{\rightarrow} \f^*.
\end{equation}
If $\Car$ is symmetric and $F$ has characteristic zero, 
this sends the basis for $[\mod{R}]$ arising from isomorphism
classes of graded self-dual irreducible $R_\al$-modules
to the dual canonical basis for $\f$. For some further details concerning Khovanov-Lauda-Rouquier categorification see \S\ref{SSKLRCat}.

Another motivation for studying representation theory of KLR algebras is the following fact first proved in \cite{BKyoung}: cyclotomic KLR algebras of finite and affine types $\mathtt A$ are explicitly isomorphic to blocks of cyclotomic Hecke algebras. The main reason this is interesting is that now we can  transport the grading from KLR algebras to cyclotomic Hecke algebras, and the resulting grading on Hecke algebras turns out to be very important, see for example \cite{BKW,BKgrdec,HM}. 

As yet another illustration, we now construct explicitly the irreducible modules for all semisimple cyclotomic Hecke algebras (both degenerate and non-degenerate). This is of course just a version of Young's orthogonal form, but the reader might  appreciate how much simpler the construction via KLR algebras is. 

We give the necessary definitions. Until the end of this subsection we assume that the Cartan matrix $\Car$ is either of type ${\tt A_\infty}$ (this is equivalent to working with sufficiently large finite type $\tt A$) or of affine type ${\tt A_{e-1}^{(1)}}$ (above we only defined the KLR algebras for finite Lie types, but the definition for ${\tt A_{e-1}^{(1)}}$ is really the same). When $\Car={\tt A_{\infty}}$, we set $e=0$ so that in both finite and affine types $\tt A$ we can  identify the set $I$ with $\Z/e\Z$. 

Fix an ordered tuple $\kappa=(k_1,\dots,k_l)\in I^l$ such that $\La=\La_{k_1}+\dots+\La_{k_l}$.  An {\em $l$-multipartition} of $d$ is an ordered $l$-tuple
of partitions $\umu = (\mu^{(1)} , \dots,\mu^{(l)})$ such that $\sum_{m=1}^l
|\mu^{(m)}|=d$.  We refer to $\mu^{(m)}$ as the $m$th component of $\umu$. 
Let $\Par^\kappa_d$ be the set of all $l$-multipartitions of $d$. Of course, $\Par^\kappa_d$ only depends on $l$, and not on $\kappa$, but as soon as we consider contents of nodes of multipartitions, the dependence on~$\kappa$ becomes  essential. The {\em Young diagram} of the multipartition $\umu = (\mu^{(1)} , \dots,\mu^{(l)})\in \Par^\ka$ is 
$$
\{(a,b,m)\in\Z_{>0}\times\Z_{>0}\times \{1,\dots,l\}\mid 1\leq b\leq \mu_a^{(m)}\}.
$$

The elements of this set
are the {\em nodes} or {\em boxes of $\umu$}. More generally, a {\em node} is any element of $\Z_{>0}\times\Z_{>0}\times \{1,\dots,l\}$. 
Usually, we identify the multipartition~$\umu$ with its 
Young diagram and visualize it as a column vector of Young diagrams. 
For example, $((3,1),\emptyset,(4,2))$ is the Young diagram
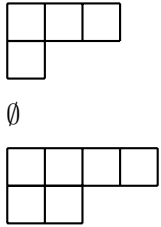
\begin{align*}
&\YoungDiagram{3,1} \\ &\emptyset\\ &\YoungDiagram{4,2}
\end{align*}

To each node $A=(a,b,m)$ 
we associate its {\em content}, which is an element of $I=\Z/e\Z$ defined as follows $\res A:=\res^\kappa A=k_m+(b-a)\pmod{e}$.
Define the {\em weight of $\umu$} to be $
\weight(\umu):=\sum_{A\in\umu}\al_{\res A} \in Q_+$.
For $\al\in Q_+$, denote
$
\Par^\kappa_\al:=\{\umu\in\Par^\kappa\mid \weight(\umu)=\al\}.
$
We call a partition {\em separating} if for any two nodes $A=(a_1,a_2,m)$, $B=(b_1,b_2,n)$ of $\umu$, we have that $\res^\kappa A=\res^\kappa B$ implies that $A$ and $B$ are on the same diagonal of the same component, i.e. $m=n$ and $a_2-a_1=b_2-b_1$. 

Let $\umu=(\mu^{(1)},\dots,\mu^{(l)})\in\Par^\ka_d$. 
A {\em $\umu$-tableau} 
$\T=(\T^{(1)},\dots,\T^{(l)})$ is obtained by 
inserting the integers $1,\dots,d$ into the boxes of $\umu$, allowing no repeats. The group $\Si_d$ acts on the set of $\umu$-tableaux from the left by acting on the entries of the tableaux. 
Let $\T^\umu$ be the $\umu$-tableau in which the numbers $1,2,\dots,d$ appear in order from left to right along the successive rows,
working from top row to bottom row. 
Set 
\begin{equation}\label{EResSeq}
\bi^\T=\bi^{\kappa,\T}=i_1^\T \dots i_d^\T\in I^d,
\end{equation}
where $i_r^\T$ is the content of the node occupied by 
$r$ in $\T$ for all $1\leq r\leq d$. 
A $\umu$-tableau $\T$ is called {\em standard}  if its entries increase from left to right along the rows and from top to bottom  along the columns within each component of $\T$. 
Let $\St(\umu)$ be the set of standard $\umu$-tableaux.

Let $\al\in Q_+$ be of height $d$ and fix a separating multipartition $\umu\in\Par^\ka_\al$. Consider a formal vector space
$
S(\umu):=\bigoplus_{\T\in\St(\umu)} F\cdot v_\T
$
on basis $\{v_\T\mid \T\in\St(\umu)\}$ labeled by the standard $\umu$-tableaux and concentrated in degree zero. Define the following action of the generators of $R_\al^\La$ on $S(\umu)$:
\begin{equation}\label{EAction}
1_\bi v_\T =\de_{\bi,\bi^\T}v_\T,\quad y_sv_\T=0,\quad
\psi_rv_\T=
\left\{
\begin{array}{ll}
v_{s_r\T} &\hbox{if $s_r\T$ is standard,}\\
0 &\hbox{otherwise.}
\end{array}
\right.
\end{equation}

\begin{Theorem} \label{TYoung}
Suppose that $\umu$ is separating. The formulas (\ref{EAction}) define a (graded) action of $R^\La_\al$ on $S(\umu)$. Moreover, $S(\umu)$ is an irreducible $R_\al$-module, and $S(\umu)\not\cong S(\unu)$ whenever $\umu\neq \unu$. 
\end{Theorem}
\begin{proof}
To prove the first statement we need to observe that the defining relations of $R_\al$ hold for the linear operators defined by (\ref{EAction}). The relations (\ref{R1})--(\ref{R2PsiE}) are clear. To see that (\ref{R6}) holds it suffices to observe that in a standard tableau, $r$ and $r+1$ can never occupy boxes on the same diagonal of the same component. As $\umu$ is separating, it follows that $\bi^\T_r\neq \bi^\T_{r+1}$ for all $\T\in\St(\umu)$, which implies (\ref{R6}). To see (\ref{R4}), if $r$ and $r+1$ occupy adjacent nodes in $\T$, then $s_r\T$ is not standard, and in this case we get $\psi_r^2 v_\T=0=Q_{i^\T_r,i^\T_{r+1}}(y_r,y_{r+1})v_\T$ as required. On the other hand, if $r$ and $r+1$ do not occupy adjacent nodes in $\T$, then $s_r\T$ is standard, 
${\tt c}_{i_r^\T,i_{r+1}^\T}=0$, 
and  $\psi_r^2 v_\T=v_\T=Q_{i^\T_r,i^\T_{r+1}}(y_r,y_{r+1})v_\T$, again as required. The relation (\ref{R3Psi}) holds trivially. Finally, to check the relation (\ref{R7}), it is enough to notice that we never have $i_r^\T=i_{r+2}^\T$ for a standard $\umu$-tableau $\T$ under the assumption that $\umu$ is separating. 

To see that $S(\umu)$ is irreducible, note first that $\Stab\neq \T$ for standard $\umu$-tableaux $\Stab$ and $\T$ implies that $\bi^\T\neq \bi^\Stab$, so acting with the idempotents $1_\bi$ yields projections to each $1$-dimensional subspace $F\cdot v_\T$ spanned by the basis elements $v_\T$. So to prove the  irreducibility of $S(\umu)$ it suffices to show that for any standard $\umu$-tableaux $\T$ and $\Stab$, there exists a series of admissible transpositions which takes $\T$ to $\Stab$, which means that there exist $1\leq k_1,\dots,k_l<d$ such that $s_{k_l}s_{k_{l-1}}\dots s_{k_1}\T=\Stab$ and $s_{k_m}s_{k_{m-1}}\dots s_{k_1}\T$ is standard for all $m=1,\dots,l$. The existence of such a sequence follows from the following: 

\vspace{1mm}
{\em Claim}. For any standard $\umu$-tableau $\T$ there exists  a series of admissible transpositions which takes $\T$ to $\T^\umu$.

\vspace{1mm}
To prove the Claim, let $A$ be the last box of the last row of $\umu$. In $\T^\umu$, the box $A$ is occupied by $d$. In $\T$, the box $A$ is occupied by some number $k\leq d$. Note that in $\T$, the numbers 
$k +1$ and $k $ do not lie on adjacent diagonals. 
So we can apply an admissible transposition to swap $k $ and $k +1$, then to swap $k +1$ and $k +2$, etc. As a result, we get a new standard $\umu$-tableau in which $A$ is occupied by $d$. Next, remove $A$ together with $d$, and apply induction. \end{proof}

The pair $(\La,d)\in P_+\times \Z_{\geq 0}$ is {\em separating} if all miltipartitions $\umu\in\Par^\kappa_d$ are separating. This notion is well-defined, since it does not depend on the choice of $\kappa=(k_1,\dots,k_l)$ such that $\La=\La_{k_1}+\dots+\La_{k_l}$. If $(\La,d)$ is separating, then all multipartitions $\umu\in\Par^\kappa_d$ have different contents, and the algebra $\bigoplus_{\al\in Q_+, \height(\al)=d}R^\La_\al$ is a semisimple algebra, with each $R^\La_\al$ being zero or simple. We have mentioned above that by the main result of \cite{BKyoung}, this algebra is isomorphic to a cyclotomic Hecke algebra $\bigoplus_{\al\in Q_+, \height(\al)=d}H^\La_\al$. This cyclotomic Hecke algebra is semisimple if and only if $(\La,d)$ is separating. Thus in all cases where a cyclotomic Hecke algebra is semisimple, Theorem~\ref{TYoung} yields an easy construction of all its irreducible representations via the isomorphism of \cite{BKyoung}. 

\section{Basic representation theory of KLR algebras}

\subsection{Semiperfect and Laurentian algebras}\label{SSGeneral}
We begin with some generalities on graded algebras. 
All gradings will be $\Z$-gradings. Let $H$ be a graded algebra over a ground field $F$. All modules, ideals, etc. are assumed to be graded, unless otherwise stated. In particular, $\rad\, V$ (resp.\ $\soc V$) is the intersection of all maximal (graded) submodules (resp. the sum of all irreducible (graded) submodules) of $V$. 
All idempotents are assumed to be degree zero. We denote by $N(H)$ the (graded) Jacobson radical of $H$.

For modules $U$ and $V$,
we write
$\hom_H(U, V)$
for homogeneous $H$-module homomorphisms, and set 
$\HOM_H(U, V):=\bigoplus_{n \in \Z} \HOM_H(U, V)_n$, 
where
$$
\HOM_H(U, V)_n := \hom_H(q^n U, V) = \hom_H(U, q^{-n}V).
$$
We define $\operatorname{ext}^d_H(U,V)$ and
$\EXT^d_H(U,V)$ similarly.
If $U$ is finitely generated, then $\HOM_H(U,V)=\Hom_H(U,V)$, where $\Hom_H(U,V)$ denoted the homomorphisms in the ungraded category. We have a similar fact for $\Ext^d$ provided $U$ has a resolution by finitely generated projective modules, in particular if $U$ is finitely generated and $H$ is Noetherian.

For an $H$-module $V$ denote by $Z(V)$ the largest   submodule of $V$ with the trivial zero degree component, i.e. $Z(V)_0=0$. Define $\overline{V}:=V/Z(V)$.

\begin{Lemma} \label{LStandard}%
 {\rm \cite{NvO}} 
 Let $V$ be an irreducible graded $H$-module, and $W$ be an irreducible $H_0$-module. 
\begin{enumerate}
\item[{\rm (i)}] If $V_n\neq 0$ for some $n\in\Z$, then $V_n$ is irreducible as an $H_0$-module. 
\item[{\rm (ii)}] The graded $H$-module $X:=\overline{H\otimes_{H_0} W}$ is irreducible, and $X_0\cong W$ as $H_0$-modules. 
\item[{\rm (iii)}] If $V_0\neq 0$, then we have $V\cong \overline{H\otimes_{H_0} V_0}$.
\end{enumerate}
\end{Lemma}
\begin{proof}
(i) is clear.

(ii) First of all note that 
$$X_0=(\overline{H\otimes_{H_0} W})_0\cong ({H\otimes_{H_0} W})_0\cong H_0\otimes_{H_0}W\cong W\neq 0.
$$
Note that $H\otimes_{H_0} W$ is generated as an $H$-module by its degree zero part $1\otimes W$, hence $X$ is also generated by its degree zero part $X_0$. Moreover, $X_0$ is irreducible as an $H_0$-module, so $X$ is generated by any non-zero vector in $X_0$. Now to prove the irreducibility of $X$ it suffices to take any homogeneous vector $v$, say of degree $n$, and prove that $H_{-n}v\neq 0$. Well, otherwise $Hv$ is a graded submodule of $X$ which avoids $X_0$, a contradiction. 

(iii) By (i) we have that $V_0$ is an irreducible $H$-module, and we have that the $H$-module $\overline{H\otimes_{H_0} W}$ is isomorphic to $V$, because it is irreducible by (ii) and  surjects onto $V$. 
\end{proof}

Now we assume that $H$ is (graded) {\em semiperfect}, i.e. every finitely generated (graded) $H$-module has a (graded) projective cover. By \cite{Das}, this is equivalent to $H_0$ being semiperfect, and is also equivalent to the fact that the following two properties hold: (1) $H/N(H)$ is (graded) semisimple Artinian; (2) idempotents lift from $H/N(H)$ to $H$. 
We fix a complete irredundant set of irreducible $H$-modules up to isomorphism and degree shift:
$$\{L(\pi)\mid \pi\in\Pi\},$$
and for each $\pi\in\Pi$, we fix a projective cover $P(\pi)$ of $L(\pi)$. 

By the semiperfectness of $H$, we have $H/N(H)$ is (graded) left Artinian, so the set $\Pi$ is finite. 
Moreover, if $\End_H(L(\pi))$ is finite dimensional over $F$ 
then by the graded version of the Wedderburn-Artin Theorem \cite[2.10.10]{NvO} the irreducible module $L(\pi)$ is finite dimensional. Finally, if $\End_H(L(\pi))=F$ for all $\pi\in \Pi$, i.e. if $H$ is {\em Schurian}, then $H/N(H)$ is a finite direct product of (graded) matrix algebras over $F$ and we have 
$$
{}_HH=\bigoplus_{\pi\in \Pi} (\DIM L(\pi))P(\pi). 
$$

A graded algebra $H$ is called {\em Laurentian} if each of its graded components $H_n$ is finite dimensional and $H_n=0$ for $n\ll0$. In this case $\DIM H$ as well as $\DIM V$ for any finitely generated $H$-module are Laurent series. 

\begin{Lemma} \label{CLaurent}
Let $H$ be a Laurentian algebra. Then:
\begin{enumerate}
\item[{\rm (i)}] $H$ has only finitely many irreducible modules up to isomorphism and degree shift;
\item[{\rm (ii)}] all irreducible $H$-modules are finite dimensional;
\item[{\rm (iii)}] $H$ is semiperfect.
\end{enumerate}  
\end{Lemma}
\begin{proof}
(i) Since $H_0$ is finite dimensional, it has only finitely many irreducible modules. It now follows from Lemma~\ref{LStandard} that up to a degree shift, $H$ has only finitely many irreducible graded modules. 

(ii) Let $V$ be an irreducible $H$-module. Then each $V_n$ is irreducible over $H_0$ by Lemma~\ref{LStandard}. So each $V_n$ is finite dimensional. On the other hand, since $V$ is cyclic and $H$ is Laurentian, $V$ has to be bounded below. Now, $V$ also has to be bounded above, since it is irreducible and $H$ is Laurentian.

(iii) follows from \cite[Theorem 3.5]{Das} since $H_0$ is semiperfect being finite dimensional. 
\end{proof}

\subsection{\boldmath Formal characters}\label{SSBasicRep} 
Fix $\al\in Q_+$ with $\height(\al)=d$. The results of the previous subsection apply to the KLR algebra $R_\al$, since it is easily seen to be Schurian, see e.g. \cite[Corollary 3.19]{KL1}, and is also  Laurentian for example in view of the following Basis Theorem: 

\begin{Theorem}\label{TBasis}{\cite[Theorem 2.5]{KL1}}, \cite[Theorem 3.7]{R} 
For each element $w\in \Si_d$ fix a reduced expression $w=s_{r_1}\dots s_{r_m}$ and set 
$
\psi_w:=\psi_{r_1}\dots \psi_{r_m}.
$ 
The elements 
$$ \{\psi_w y_1^{m_1}\dots y_d^{m_d}1_\bi \mid w\in \Si_d,\ m_1,\dots,m_d\in\Z_{\geq 0}, \ \bi\in \words_\al\}
$$ 
form an $F$-basis of  $R_\al$. 
\end{Theorem}

There exists a homogeneous algebra anti-involution 
\begin{equation}\label{ECircledast}
\tau:R_\al\longrightarrow R_\al,\quad 1_\bi\mapsto 1_\bi,\quad y_r\mapsto y_r,\quad \psi_s\mapsto \psi_s  
\end{equation}
for all $\bi\in \words_\al,\ 1\leq r\leq d$, and $1\leq s<d$. If $M=\bigoplus_{d\in\Z}M_d$ is a finite dimensional 
graded $R_\al$-module, then the {\em graded dual}
$M^\circledast$ is the graded $R_\al$-module such that $(M^\circledast)_n:=\Hom_F(M_{-n},F)$, for all
$n\in\Z$, and the $R_\al$-action is given by $(xf)(m)=f(\tau(x)m)$, for all $f\in M^\circledast, m\in M, x\in
R_\al$.

For every irreducible module $L$, there is a unique choice of the grading shift so that we have $L^\circledast \cong L$ \cite[\S3.2]{KL1}. When speaking of irreducible $R_\al$-modules we often assume by fiat that the shift has been chosen in this way. 

For $\bi\in \words_\al$ and $M\in\mod{R_\al}$, the {\em $\bi$-word space} of $M$ is
$
M_\bi:=1_\bi M.
$
We have the word space decomposition: 
$$
M=\bigoplus_{\bi\in \words_\al}M_\bi.
$$
We say that $\bi$ is a {\em word of $M$} if $M_\bi\neq 0$.
Note from the relations that 
$
\psi_r M_\bi\subset M_{s_r \bi}.
$
Define the {\em  (graded formal) character} of $M$ as follows: 
\begin{equation*}
\CH M:=\sum_{\bi\in \words_\al}(\DIM M_\bi) \bi \in \A\words_\al.
\end{equation*}
The  character map $\CH: \mod{R_\al}\to \A\words_\al$ factors through to give an {\em injective} $ \A$-linear map 
$
\CH: [\mod{R_\al}]\to  \A\words_\al, 
$ see \cite[Theorem 3.17]{KL1}. 

Let $\bi=i_1\dots i_d$ and $\bj=i_{d+1}\dots i_{d+f}$ be two elements of $\words$. 
 Define the {\em quantum shuffle product}: 
 \begin{equation*}
\bi\circ\bj:=\sum q^{-e(\sigma)}i_{\sigma(1)}\dots i_{\sigma(d+f)} \in\A\words,
\end{equation*}
where the sum is over all $\sigma\in \Si_{d+f}$ such that $\sigma^{-1}(1)<\dots<\sigma^{-1}(d)$ and $\sigma^{-1}(d+1)<\dots<\sigma^{-1}(d+f)$, and 
$$
e(\sigma):=\sum_{k\leq d<m,\ \sigma^{-1}(k)>\sigma^{-1}(m)} \cc_{i_{\si(k)}, i_{\sigma(m)}}.
$$ 
This defines an $\A$-algebra structure on the $\A$-module $\A\words$, which consists of all finite formal $\A$-linear combinations of elements $\bi\in \words$.

In view of \cite[Lemma 2.20]{KL1}, we have
\begin{equation}\label{ECharShuffle}
\CH(M_1\circ\dots\circ M_n)=\CH(M_1)\circ\dots\circ \CH(M_n).
\end{equation}

\subsection{Crystal operators and extremal words}\label{SSCOES}

The theory of crystal operators has been developed in \cite{KL1}, \cite{LV} and \cite{KK} following ideas of Grojnowski \cite{Gr}, see also \cite{Kbook}. We review necessary facts for reader's convenience. 

Let $\al\in Q_+$ and $i\in I$. By Example~\ref{ExNH}, $R_{n\al_i}$ is a nil-Hecke algebra with unique irreducible module $L(\al_i^n)$ with $\DIM L(\al_i^n)=[n]^!_{i}$. 
We have functors 
\begin{align*}
&e_i: \mod{R_\al}\to\mod{R_{\al-\al_i}},\ M\mapsto \Res^{R_{\al-\al_i,\al_i}}_{R_{\al-\al_i}}\circ \Res_{\al-\al_i,\al_i}M,
\\
&f_i: \mod{R_\al}\to\mod{R_{\al+\al_i}},\ M\mapsto \Ind_{\al,\al_i}M\boxtimes L(\al_i).
\end{align*}
If $L\in\mod{R_\al}$ is irreducible, we define
$$
\tilde f_i L:=\head (f_i L),\quad \tilde e_i L:=\soc (e_i L).
$$
A fundamental fact is that $\tilde f_i L$ is again irreducible and $\tilde e_i L$ is irreducible or zero. 
We refer to $\tilde e_i$ and $\tilde f_i$ as the crystal operators. These are operators on $B\cup\{0\}$, where $B$ is the set of isomorphism classes of  irreducible $R_\al$-modules for all $\al\in Q_+$.  
Define
$
\wt:B\to P,\ [L]\mapsto -\al
%
$
if $L\in\mod{R_\al}$.

\begin{Theorem} \label{TLV} {\rm \cite{LV}} 
The set $B$ with the operators $\tilde e_i,\tilde f_i$ and the function $\wt$ is the crystal graph of the negative part $U_q(\n_-)$ of the quantized enveloping algebra of $\g$ of Lie type $\Car$.  
\end{Theorem}

For any $M\in\mod{R_\al}$, we define
$$
\eps_i(M):=\max\{k\geq 0\mid e_i^k(M)\neq 0\}.
$$
Then $\eps_i(M)$ is also the length of the longest `$i$-tail' of words of $M$, i.e. the maximum of $k\geq 0$ such that $j_{d-k+1}=\dots=j_d=i$ for some word $\bj=j_1\dots j_d$ of $M$.

\begin{Proposition} \label{PCryst1} {\rm \cite{LV,KL1}} 
Let $L$ be an irreducible $R_\al$-module, $i\in I$, and $\eps=\eps_i(L)$. 
\begin{enumerate}
\item[{\rm (i)}] $\tilde e_i\tilde f_iL\cong L$ and if $\tilde e_i L\neq 0$ then $\tilde f_i\tilde e_iL\cong L$; 
\item[{\rm (ii)}] $\eps=\max\{k\geq 0\mid \tilde e_i^k(L)\neq 0\}$;
\item[{\rm (iii)}] $\Res_{\al-\eps\al_i,\eps\al_i}L\cong \tilde e_i^\eps L\boxtimes L(\al_i^\eps)$.
\end{enumerate}
\end{Proposition}

Let $i\in I$. Consider the map  
$
\theta_i^*:\words\to\words
$ such that for $\bj=j_1\dots j_d \in\words$, we have
\begin{equation}\label{ETheta*}
\theta_i^*(\bj)=
\left\{
\begin{array}{ll}
j_1,\dots,j_{d-1} &\hbox{if $j_d=i$;}\\
0 &\hbox{otherwise.}
\end{array}
\right.
\end{equation}
We extend $\theta_i^*$ by linearity to a map $\theta_i^*:\A\words\to\A\words$. 

Let $x$ be an element of $\A\words$. Define 
$$\eps_i(x):=\max\{k\geq 0\mid (\theta_i^*)^{k}(x)\neq 0\}.$$ 
A word $i_1^{a_1}\dots i_b^{a_b}\in\words$, with  $a_1,\dots,a_b\in\Z_{\geq 0}$, is called {\em extremal} for $x$ if 
$$a_b=\eps_{i_b}(x),\ a_{b-1}=\eps_{i_{b-1}}((\theta_{i_b}^*)^{a_b}(x))\ ,\ \dots\ ,\ a_1=\eps_{i_1}\big((\theta_{i_2}^*)^{a_2}\dots (\theta_{i_b}^*)^{a_b}(x)\big).$$ 
A word $i_1^{a_1}\dots i_b^{a_b}\in\words_\al$ is called {\em extremal} for $M\in\mod{R_\al}$ if it is an extremal word for $\CH M\in\A\words$, in other words, if  
$$a_b=\eps_{i_b}(M),\ a_{b-1}=\eps_{i_{b-1}}(\tilde e_{i_b}^{a_b}M)\ ,\ \dots\ ,\  a_1=\eps_{i_1}(\tilde e_{i_2}^{a_2}\dots\tilde e_{i_b}^{a_b}M).
$$   

The following useful result, which is a version of \cite[Corollary 2.17]{BKdurham}, describes the multiplicities of extremal word spaces in irreducible modules. We denote by $1_F$ the trivial module $F$ over the trivial algebra $R_0\cong F$.

\begin{Lemma} \label{LExtrMult}
Let $L$ be an irreducible $R_\al$-module, and $\bi=i_1^{a_1}\dots i_b^{a_b}\in\words_\al$ be an extremal word for $L$. Then 
$\DIM L_\bi=[a_1]^!_{i_1}\dots [a_b]^!_{i_b}$, and 
$$L\cong \tilde f_{i_b}^{a_b} \tilde f_{i_{b-1}}^{a_{b-1}}\dots\tilde f_{i_1}^{a_1}1_F.$$  Moreover, $\bi$ is not an extremal word for any irreducible module $L'\not\cong L$. 
\end{Lemma}
\begin{proof}
Follows easily from Proposition~\ref{PCryst1}, cf. \cite[Theorem 2.16]{BKdurham}. 
\end{proof}

\begin{Corollary} \label{CExtrNew}
Let $M\in\mod{R_\al}$, and $\bi=i_1^{a_1}\dots i_b^{a_b}\in\words_\al$ be an extremal word for $M$. Then we can write $\DIM M_\bi=m[a_1]^!_{i_1}\dots [a_b]^!_{i_b}$ for some $m\in\A$. Moreover, if $L\cong \tilde f_{i_b}^{a_b} \tilde f_{i_{b-1}}^{a_{b-1}}\dots\tilde f_{i_1}^{a_1}1_F$ and $L^\circledast\cong L$, then we have $[M:L]_q=m$. 
\end{Corollary}
\begin{proof}
Apply Lemma~\ref{LExtrMult}, cf. \cite[Corollary 2.17]{BKdurham}. 
\end{proof}

Now we establish some useful `multiplicity-one results'. The first one shows that in every irreducible module there is a word space with a one dimensional graded component:

\begin{Lemma} \label{LMultOneWeight}
Let $L$ be an irreducible $R_\al$-module, and $\bi=i_1^{a_1}\dots i_b^{a_b}\in\words_\al$ be an extremal word for $L$. Set 
$N:=\sum_{m=1}^b a_m(a_m-1)(\al_{i_m},\al_{i_m})/4.$ 
Then $\dim 1_\bi L_N=\dim 1_\bi L_{-N}=1$. 
\end{Lemma}
\begin{proof}
This follows immediately from the equality $\DIM 1_\bi L=[a_1]^!_{i_1}\dots [a_b]^!_{i_b}$, which comes from Lemma~\ref{LExtrMult}. 
\end{proof}

The following result shows that any induction product of irreducible modules always has a multiplicity one composition factor.

\begin{Proposition} \label{PProdIrrMult1}
Suppose that $n\in\Z_{>0}$ and for $r=1,\dots,n$, we have $\al^{(r)}\in Q_+$, an irreducible $R_{\al^{(r)}}$-module $L^{(r)}$, and  $\bi^{(r)}:=i_1^{a^{(r)}_1}\dots i_k^{a^{(r)}_k}\in\words_{\al^{(r)}}$ is an extremal word for $L^{(r)}$. Denote $a_t:=\sum_{r=1}^na^{(r)}_t$ for all $1\leq t\leq k$. 
Then 
$
\bj:=i_1^{a_1}\dots i_k^{a_k}
$
is an extremal word for $L^{(1)}\circ \dots \circ L^{(n)}$, 
and the graded multiplicity of the $\circledast$-self-dual irreducible module 
$$N\simeq\tilde f_{i_k}^{a_k} \tilde f_{i_{k-1}}^{a_{k-1}}\dots\tilde f_{i_1}^{a_1}1_F$$ 
in $L^{(1)}\circ \dots \circ L^{(n)}$ is $q^{m}$,
where
$$m:=-\textstyle\sum_{1\leq t<u\leq n}\left(\sum_{1\leq r< s\leq k}a_r^{(u)}a_s^{(t)}(\al_{i_r},\al_{i_s})+\frac{1}{2}\sum_{r=1}^k a_r^{(t)}a_r^{(u)}(\al_{i_r},\al_{i_r})\right). 
$$
In particular, the ungraded multiplicity of $N$ in $L^{(1)}\circ \dots \circ L^{(n)}$ is one.  
\end{Proposition}
\begin{proof}
By Lemma~\ref{LExtrMult}, the multiplicity of $\bi^{(r)}$ in $\CH L^{(r)}$ is  $[a^{(r)}_1]_{i_1}^{!}\dots [a^{(r)}_k]_{i_k}^!$. By (\ref{ECharShuffle}), we have 
$$\CH(L^{(1)}\circ \dots \circ L^{(n)})=\CH(L^{(1)})\circ \dots \circ \CH(L^{(n)}).$$ 
It is easy to see that the word $\bj$  
is an extremal word for $L^{(1)}\circ \dots \circ L^{(n)}$, and that $\bj$ can be obtained only from the shuffle product $\bi^{(1)}\circ \dots\circ \bi^{(n)}$. An elementary  computation shows that $\bj$ appears in $\bi^{(1)}\circ \dots\circ \bi^{(n)}$ with multiplicity  $q^{m}[a_1]_{i_1}^!\dots [a_k]_{i_k}^!$. 
Now apply Corollary~\ref{CExtrNew}. 
\end{proof}


\begin{Corollary} \label{CPowerIrr}
Let $L$ be an irreducible $R_\al$-module and $n\in\Z_{>0}$. Then there is an irreducible $R_{n\al}$-module $N$ which appears in $L^{\circ n}$ with graded multiplicity $q_\al^{-n(n-1)/2}$. In particular, the ungraded multiplicity of $N$ is one.  
\end{Corollary}
\begin{proof}
Apply Proposition~\ref{PProdIrrMult1} with $L^{(1)}=\dots= L^{(n)}=L$. 
\end{proof}

\subsection{Khovanov-Lauda-Rouquier categorification}\label{SSKLRCat}
We recall the Khovanov-Lauda-Rouquier categorification of the quantized enveloping algebra $\f$ obtained in \cite{KL1,KL2,R}, and briefly mentioned in \S\ref{SMotiv}. 
Let $\f_\A\subset \f$ be the $\A$-form of the Lusztig's quantum group $\f$ corresponding to the Cartan matrix $\Car$. This $\A$-algebra is generated by the divided powers $\theta_i^{(n)}=\theta_i^n /[n]_{i}^!$ of the standard generators. 
The algebra $\f_\A$ has a $Q_+$-grading $\f_\A=\oplus_{\al\in Q_+}(\f_\A)_\al$ so that each $\theta_i$ is in degree $\al_i$. 

There is a bilinear form $(\cdot,\cdot)$ on $\f$ defined in \cite[$\S$1.2.5, $\S$33.1.2]{Lubook}. Let $\f_\A^*= \left\{y \in \f\:\big|\:(x,y)\in\A\text{ for all }x \in \f_\A\right\}$. Let $(\theta_i^*)^{(n)}$ be the map dual to the map $\mathbf{f}_\A\to \mathbf{f}_\A,\ x\mapsto x\theta_i^{(n)}$. Finally, there is a coproduct $r$  on $\f$ such that $\f$ is a twisted unital and counital bialgebra. Moreover, for all $x,y,z\in \f$ we have 
\begin{equation}\label{EDefPropForm}
(xy,z)=(x\otimes y,r(z)).
\end{equation}

The field $\Q(q)$ possesses a unique automorphism called the
{\em bar-involution} such that $\overline{q} = q^{-1}$.
With respect to this involution, 
let $\barinv:\f \rightarrow \f$
be the anti-linear algebra automorphism
such that $\barinv(\theta_i) = \theta_i$ for all $i \in I$.
Also let 
$\barinv^*:\f \rightarrow \f$
be the adjoint anti-linear
map to $\barinv$ 
with respect to Lusztig's form, so 
$(x, \barinv^*(y)) = \overline{(\barinv(x), y)}$
for all $x, y \in \f$.
The maps $\barinv$ and $\barinv^*$ preserve $\f_\A$ and $\f_\A^*$, respectively. 

Let $[\mod{R}] = \bigoplus_{\alpha \in Q_+} [\mod{R_\alpha}]$
denote the Grothendieck ring, which is an $\A$-algebra via induction product. Similarly the functors of restriction define a coproduct $r$  on $[\mod{R}]$. This product and coproduct make $[\mod{R}]$ into a twisted unital and counital bialgebra \cite[Proposition 3.2]{KL1}.

 In \cite{KL1,KL2} an explicit $\A$-bialgebra isomorphisms
$
\ga^*:[\mod{R}] \stackrel{\sim}{\rightarrow}
\f_\A^*
$
is constructed (this has already been mentioned in (\ref{EGamma*})). In fact \cite{KL1} establishes a dual isomorphism $\ga$,  see \cite[Theorem 4.4]{KR2} for all details on this. Moreover, $\ga^*([V^\circledast])=\barinv^*(\ga^*([V]))$, and we have a commutative triangle
\begin{equation}\label{ETriangle}
\begin{pb-diagram}
\node{}\node{\A\words}
\node{} \\
\node{[\mod{R}]} \arrow[2]{e,t}{\ga^*}
\arrow{ne,t}{\CH}
\node{}\node{\mathbf{f}_\A^*}
\arrow{nw,t}{\iota}
\end{pb-diagram},
\end{equation}
where the map $\iota$ is defined as follows: 
$$\iota(x)=\sum_{\bi=i_1 \dots i_d\in\words}(x,\theta_{i_1}\dots\theta_{i_d})\bi\qquad(x\in \mathbf{f}_\A^*).$$ 

\begin{Lemma} \label{LExtrDCB}
Let $v^*$ be a dual canonical basis element of $\f$, and $\bi=i_1^{a_1}\dots i_k^{a_k}$ be an extremal word of $\iota(v^*)$ in the sense of \S\ref{SSCOES}. Then $\bi$ appears in $\iota(v^*)$  with coefficient $[a_1]_{i_1}^!\dots[a_k]_{i_k}^!$. 
\end{Lemma}
\begin{proof}
Apply induction on $a_1+\dots+a_k$. The induction base is  $a_1+\dots+a_k=0$, in which case $v^*=1\in\f^*_\A$ and $\iota(1)$ is the empty word. 
Recall the map $\theta_i^*:\A\words\to\A\words$ from (\ref{ETheta*}). For all $x\in\f^*_\A$ we have 
$
\iota((\theta_i^*)^{(n)}(x))=(\theta_i^*)^{(n)}(\iota(x))
$, where in the right hand side $(\theta_i^*)^{(n)}=(\theta_i^*)^{n}/[n]_{\al_i}^!$. 
By \cite[Proposition 5.3.1]{KaG}, 
$(\theta_{i_k}^*)^{(a_{i_k})} (v^*)$
is again a dual canonical basis element, and by induction, the word $i_1^{a_1}\dots i_{k-1}^{a_{k-1}}$ appears in $\iota((\theta_{i_k}^*)^{(a_{i_k})} (v^*))$ with coefficient $[a_1]_{i_1}^!\dots[a_{k-1}]_{i_{k-1}}^!$. The result follows. 
\end{proof}



\section{Standard module theory}
We want to classify the irreducible $R_\al$-modules using a standard module theory. This was first done in  \cite{KR2} and then substantially developed and generalized in \cite{McN}. Here we mainly follow the approach of \cite{McN}, with an occasional idea from \cite{Kcusp}. 

\subsection{Convex orders and cuspidal systems}
The theory depends on a choice of a convex order on the set $\Phi_+$ of {\em positive roots} (we always mean the system of  positive roots corresponding to our fixed choice of the simple roots $\al_i$). 
We also denote by $W$  the Weyl group of the {\em root system} $\Phi$. It is a Coxeter group with standard generators $\{r_i\mid i\in I\}$. 

A {\em convex order} on $\Phi_+$ is a total order $\prec$
such that
$$
\beta, \gamma, 
\beta+\gamma \in \Phi_+,\ \beta \prec \gamma \quad\Rightarrow \quad
\beta \prec \beta + \gamma \prec \gamma.
$$
By \cite{Papi}, there is a bijection between convex orders on  
$\Phi_+$ and reduced expressions for the longest element $w_0$ of $W$ which works as follows: 
given a reduced expression
$w_0 = r_{i_1} \cdots r_{i_N}$ the corresponding 
convex order on $\Phi_+$ is given by
$$
\alpha_{i_1} \prec
r_{i_1}(\alpha_{i_2})
\prec
r_{i_1} r_{i_2}(\alpha_{i_3})
\prec\cdots\prec r_{i_1} \cdots r_{i_{N-1}}(\alpha_{i_N}).
$$
This allows one to prove the following easy lemma, see 
\cite[Lemma 2.4]{BKM}:

\begin{Lemma}\label{l1}
Suppose we are given positive roots $\alpha, \beta_1,\dots,\beta_k, \gamma_1,\dots,\gamma_l$
such that $\beta_i \preceq \alpha \preceq \gamma_j$
for all $i$ and $j$.
We have that 
$\beta_1+\cdots+\beta_k = \gamma_1+\cdots + \gamma_l$
if and only if $k=l$
 and
$\beta_1=\cdots=\beta_k=\gamma_1=\cdots=\gamma_l = \alpha$.
\end{Lemma}

Now we give a key definition:

\begin{Definition} \label{DCus}
{\rm 
A {\em cuspidal system} (for a fixed convex preorder) is the following data: an irreducible $R_\rho$-module $L_\rho$ assigned to every positive $\rho\in \Phi_+^\re$, with the following property: 

\vspace{1mm}
\noindent
(Cus) if $\be,\ga\in Q_+$ are non-zero elements such that $\rho=\be+\ga$ and $\Res_{\be,\ga}L_\rho\neq 0$, then $\beta$ is a sum of positive roots less than $\rho$ and $\ga$ is a sum of positive roots greater than $\rho$. 
}
\end{Definition}

\vspace{1mm}
It is not obvious that a cuspidal system exists or is unique (for a fixed convex order). This will be proved later. 

Let us fix a convex order $\prec$ on $\Phi_+$ and an element $\al\in Q_+$. A {\em root partition} of $\al$ is a weakly decreasing tuple 
\begin{equation}\label{ERP}
(\be_1\succeq\be_2\succeq\dots\succeq \be_n)
\end{equation} 
of positive roots such that $\be_1+\be_2+\dots+ \be_n=\al$. 
The set of root partitions of $\al$ is denoted by $\Pi(\al)$. 
For example, if $\rho$ is a positive root, there always is a trivial root partition $(\rho)\in\Pi(\rho)$. 

Sometimes we use other notations for  root partitions. 
Let $$\rho_1\succ\dots\succ\rho_N$$ be all positive roots taken in decreasing order. Collecting together equal terms of the root partition (\ref{ERP}), we can write it as 
$$
\pi=(\rho_1^{m_1},\dots,\rho_N^{m_N})
$$
with $\sum_{n=1}^Nm_n\rho_n$ or simply as a tuple 
$$\pi=(m_1,\dots,m_N)$$ of nonnegative integers such that $\al=\sum_{n=1}^Nm_n\rho_n$.

The left lexicographic order on $\Pi(\al)$ is denoted $\leq_l$ and the right lexicographic order on $\Pi(\al)$ is denoted $\leq_r$. We will also use the following {\em bilexicographic} partial order on $\Pi(\al)$: 
$$
\pi\leq \si\qquad\text{if and only if}\qquad \pi\leq_l \si \ \text{and}\ \pi\geq_r \si.
$$

Let $\g$ be the finite dimensional complex semisimple Lie algebra with Cartan matrix $\Car$. The positive subalgebra $\n_+\subset \g$ has a basis consisting of {\em root vectors} $\{E_\rho\mid \rho\in\Phi_+\}$. To a root partition $\pi=(m_1,\dots,m_N)$, we assign a PBW monomial 
$
E_{\pi}:=E_{\rho_1}^{m_1}\dots E_{\rho_N}^{m_N}.$
Then $\{E_{\pi}\mid\pi\in\Pi(\al)\}$ is a basis of 
 the weight space $U(\n_+)_\al$. In particular, $|\Pi(\al)|=\dim U(\n_+)_\al$. In view of the isomorphism $\ga^*$ from (\ref{EGamma*}), we conclude:

\begin{Lemma} \label{LAmount}
The number of irreducible $R_\al$-modules (up to isomorphism and degree shift) is $|\Pi(\al)|$.  
\end{Lemma}

\subsection{Standard modules}
We continue to work with a fixed convex preorder $\prec$ on $\Phi_+$.  Let $\{L_\rho \mid \rho\in\Phi_+\}$ be a cuspidal system for $\prec$ (we are yet to prove that it exists!). Fix $\al\in Q_+$ and a root partition $\pi=(m_1,\dots,m_N)\in\Pi(\al)$. We define an integer
\begin{equation}\label{EShift}
\shift(\pi):=\sum_{k=1}^N (\rho_k,\rho_k)m_k(m_k-1)/4.
\end{equation}
Set 
$$
|\pi|=(m_1\rho_1,\dots,m_N\rho_N)\in Q_+^N.
$$
The corresponding parabolic subalgerba is
$$
R_{|\pi|}=R_{m_1\rho_1,\dots,m_N\rho_N}\cong R_{m_1\rho_1}\otimes\dots\otimes R_{m_N\rho_N}\subseteq R_\al.
$$

Next, we define the $R_{|\pi|}$-module
\begin{equation}
L_\pi:=q^{\shift(\pi)}\,L_{\rho_1}^{\circ m_1} \boxtimes \dots\boxtimes L_{\rho_N}^{\circ m_N} , 
\end{equation}
and we define the {\em proper standard module}
\begin{equation}\label{EStand}
\Stand(\pi):= \Ind_{|\pi|} L_{\pi}\cong q^{\shift(\pi)}\,L_{\rho_1}^{\circ m_1} \circ \dots\circ L_{\rho_N}^{\circ m_N} \in\mod{R_\al}.
\end{equation}
Also introduce the {\em proper costandard module}
\begin{equation}\label{EBarNabla}
\bar\nabla(\pi) := \bar\Delta(\pi)^\circledast.
\end{equation}

It will become clear in Lemma~\ref{LLMMuSeldDual} why we apply the shift by $\shift(\pi)$ in our definitions. 

\begin{Lemma} \label{LDualInd}
Let $\underline{\ga}:=(\ga_1,\dots,\ga_n)\in Q_+^n$, and $V_m\in\mod{R_{\ga_m}}$ for $m=1,\dots,n$. 
Denote
$
d(\underline{\ga})=\sum_{1\leq m<k\leq n}(\ga_m,\ga_k).
$ 
Then 
$
(V_1\circ\dots\circ V_n)^\circledast\cong 
q^{d(\underline{\ga})}(V_n^\circledast\circ\dots\circ V_1^\circledast).
$
\end{Lemma}
\begin{proof}
Follows from \cite[Theorem 2.2]{LV} by uniqueness of adjoint functors as in the proof of \cite[Corollary 3.7.4]{Kbook}
\end{proof}

Recall that for every irreducible module $L$, there is a unique choice of the grading shift so that we have $L^\circledast \cong L$, and unless otherwise stated,  we assume that the shift has been chosen in this way. This in particular applies to the modules $L_\rho$ of our cuspidal system.

\begin{Lemma} \label{LPowerCuspSelfDual}
Let $\rho\in\Phi_+^\re$, $L_\rho$ be the corresponding cuspidal module, and $n\in \Z_{>0}$. Then 
$$
(L_\rho^{\circ n})^\circledast\cong 
q_\rho^{n(n-1)}
L_\rho^{\circ n}.$$
In particular, the module $q_\rho^{n(n-1)/2} L_\rho^{\circ n}$ is $\circledast$-self-dual. 
\end{Lemma}
\begin{proof}
Recall that our standard choice of shifts of irreducible modules is so that $L_\rho^{\circledast}\cong L_\rho$. Now the result follows from Lemma~\ref{LDualInd}. 
\end{proof}

\begin{Lemma} \label{LLMMuSeldDual}
We have $L_{\pi}^\circledast\cong L_{\pi}$
\end{Lemma}
\begin{proof}
Recall that our standard choice of shifts of irreducible modules is so that $L_\rho^{\circledast}\cong L_\rho$. Let $\rho\in\Phi_+$ and $n\in \Z_{>0}$. Then by Lemma~\ref{LPowerCuspSelfDual}, we have that the module $q_\rho^{n(n-1)/2}L_\rho^{\circ n}$ is $\circledast$-self-dual. 
The result follows. 
\end{proof}

\subsection{Restrictions of proper standard modules}
We recall the Mackey Theorem of Khovanov and Lauda \cite[Proposition~2.18]{KL1}. 
Given $x\in \Si_n$ and $\underline{\ga}=(\ga_1,\dots,\ga_n)\in Q_+^n$, we denote 
\begin{align*}
x\underline{\ga}&:=(\ga_{x^{-1}(1)},\dots,\ga_{x^{-1}(n)})\in Q_+^n. 
\\
s(x,\underline{\ga})&:=-\sum_{1\leq m<k\leq n,\ x(m)>x(k)}(\ga_m,\ga_k)\in\Z.
\end{align*}

Writing $R_{\underline{\ga}}$ for $R_{\ga_1,\dots,\ga_n}$, there is an obvious natural algebra isomorphism 
$$
\phi^x:R_{x\underline{\ga}}\to R_{\underline{\ga}}
$$
permuting the components. Composing with this isomorphism, we get a functor
$$
\mod{R_{\underline{\ga}}}\to \mod{R_{x\underline{\ga}}},\  M\mapsto {}^{\phi^x}M.
$$
Making an additional shift, we get a functor 
\begin{equation}\label{ETwist}
\mod{R_{\underline{\ga}}}\to \mod{R_{x\underline{\ga}}},\  M\mapsto {}^xM:=q^{s(x,\underline{\ga})}({}^{\phi^x}M).
\end{equation}

\begin{Theorem} \label{TMackeyKL}
Let $\underline{\ga}=(\ga_1,\dots,\ga_n)\in Q_+^n$ and $\underline{\be}=(\be_1,\dots,\be_m)\in Q_+^m$ with $\ga_1+\dots+\ga_n=\be_1+\dots+\be_m=:\al$. Then for any $M\in\mod{R_{\underline{\ga}}}$ we have that $\Res_{\underline{\be}}\,\Ind_{\underline{\ga}} M$ has  filtration with factors of the form
$$
\Ind_{\al^{1}_1,\dots,\al^{n}_{1}\,;\,\dots\,;\,\al^{1}_{m},\dots,\al^{n}_{m}}^{\,\be_1\,;\,\dots\,;\,\be_m}
{}^{x(\underline{\al})}\big(\Res_{\al^{1}_1,\dots,\al^{1}_{m}\,;\,\dots\,;\,\al^{n}_{1},\dots,\al^{n}_{m}}^{\,\ga_1\,;\,\dots\,;\,\ga_n}
\,M \big)
$$
with $\underline{\al}=(\al^a_b)_{1\leq a\leq n,\ 1\leq b\leq m}$ running over all tuples of elements of $Q_+$ such that  $\sum_{b=1}^m\al^{a}_b=\ga_a$ for all $1\leq a\leq n$ and $\sum_{a=1}^n\al^{a}_b=\be_b$ for all $1\leq b\leq m$, and $x(\underline{\al})$ is the permutation of $mn$ which maps 
$$
(\al^{1}_1,\dots,\al^{1}_{m};\al^{2}_1,\dots,\al^{2}_{m};\dots;\al^{n}_{1},\dots,\al^{n}_{m})
$$
to
$$
(\al^{1}_1,\dots,\al^{n}_{1};\al^{1}_2,\dots,\al^{n}_{2};\dots;\al^{1}_{m},\dots,\al^{n}_{m}).
$$
\end{Theorem}

We use the Mackey Theorem to study restrictions of proper standard modules:

\begin{Proposition}\label{P1}
Let $\pi,\si\in\Pi(\al)$. Then:
\begin{enumerate}
\item[{\rm (i)}] $\Res_{|\si|} \Stand(\si)\cong L_{\si}$.
\item[{\rm (ii)}] $\Res_{|\pi|} \Stand(\si)\neq 0$ implies $\pi\leq \si$. 
\end{enumerate}
\end{Proposition}
\begin{proof}
Write $\pi=(m_1,\dots,m_N)$, $\si=(n_1,\dots, n_N)$. 
Let $\Res_{|\pi|} \Stand(\si)\neq 0$. It suffices to prove that $\pi\geq_l \si$ or $\pi\leq_r \si$ implies that $\pi=\si$ and  $\Res_{|\pi|} \Stand(\si)\cong L_{\si}$. We may assume that $\pi\geq_l \si$, the case $\pi\leq _r \si$ being similar. We apply induction on $\height(\al)$. 
Pick the minimal $a$ with $m_a\neq 0$. Let 
$$\pi'=(0,\dots,0,m_{a+1},\dots, m_N)\in\Pi(\al-m_a\rho_a)$$ 
and 
$$\si'=(0,\dots,0,n_{a+1},\dots, n_N)\in\Pi(\al-n_a\rho_a).$$ 
By Theorem~\ref{TMackeyKL}, $\Res_{|\pi|} \Stand(\si)$ has  filtration with factors of the form 
$
\Ind^{m_a\rho_a;|\pi'|}_{\kappa_1,\dots,\kappa_c;\underline{\ga}}V,
$
where $m_a\rho_a=\kappa_1+\dots+\kappa_c$, with $\kappa_1,\dots,\kappa_c\in Q_+\setminus\{0\}$, and $\underline{\ga}$ is a refinement of $|\pi'|$. Moreover, the module $V$ is obtained by twisting and degree shifting as in (\ref{ETwist}) of a module obtained by restriction of 
$
L_{\rho_1}^{\boxtimes n_1}\boxtimes\dots\boxtimes  L_{\rho_N}^{\boxtimes n_N}
$
to a parabolic which has $\kappa_1,\dots,\kappa_c$ in the beginnings of the corresponding blocks. In particular, if $V\neq 0$, then for each $b=1,\dots,c$ we have that $\Res_{\kappa_b,\rho_k-\kappa_b}L_{\rho_k}\neq 0$ for some $k=k(b)$ with $n_k\neq 0$. 

Let $1\leq b\leq c$. If $\Res_{\kappa_b,\rho_k-\kappa_b}L_{\rho_k}\neq 0$, then by the definition of cuspidal modules, $\kappa_b$ is a sum of roots $\preceq \rho_k$. Moreover, since $\pi\geq_l \si$ and $n_k\neq0$, we have that $\rho_k\preceq\rho_a$. Thus $\kappa_b$ is a sum of roots $\preceq \rho_a$. Using Lemma~\ref{l1}, we conclude that $c=m_a$ and $\kappa_b=\rho_a=\rho_{k(b)}$ for all $b=1,\dots,c$. Hence $n_a\geq m_a$. Since  $\pi\geq _l \si$, we conclude that $n_a=m_a$, and 
$$
\Res_{|\pi|} \Stand(\si)\simeq L_{\rho_a}^{\circ m_a}\boxtimes \Res_{|\pi'|} \Stand(\si'). 
$$
Since $\height(\al-m_a\rho_a)<\height(\al)$, we can apply the inductive hypothesis. 
\end{proof}

\subsection{Classification of irreducible modules}\label{SRough}
We continue to work with a fixed convex preorder $\preceq$ on $\Phi_+$.  
In this subsection we prove the following theorem:

\begin{Theorem} \label{THeadIrr} 
For a given convex preorder, there exists a unique cuspidal system $\{L_\rho\mid \rho\in \Phi_+\}$. Moreover: 
\begin{enumerate}
\item[{\rm (i)}] For every root partition $\pi$, the proper standard module  
$
\Stand(\pi)
$ has an irreducible head; denote this irreducible module $L(\pi)$. 

\item[{\rm (ii)}] $\{L(\pi)\mid \pi\in \Pi(\al)\}$ is a complete and irredundant system of irreducible $R_\al$-modules up to isomorphism.

\item[{\rm (iii)}] $L(\pi)^\circledast\cong L(\pi)$.  

\item[{\rm (iv)}] $[\Stand(\pi):L(\pi)]_q=1$, and $[\Stand(\pi):L(\si)]_q\neq 0$ implies $\si\leq \pi$. 

\item[{\rm (v)}] $\Res_{|\pi|}L(\pi)\cong L_{\pi}$ and $\Res_{|\si|}L(\pi)\neq 0$ implies $\si\leq \pi$.  

\item[{\rm (vi)}] $L_\rho^{\circ n}$ is irreducible for all $\rho\in \Phi_+$ and all $n\in\Z_{>0}$.  
\end{enumerate}
\end{Theorem}

The rest of \S\ref{SRough} is devoted to the proof of Theorem~\ref{THeadIrr}, which goes by induction on $\height(\al)$. To be more precise, we prove the following statements for all $\al\in Q_+$ by induction on $\height(\al)$: 
\begin{enumerate}
\item[{\rm (1)}] For each $\rho\in\Phi_+$ with $\height(\rho)\leq\height(\al)$ there exists a unique up to isomorphism irreducible $R_\rho$-module $L_\rho$ which satisfies the property (Cus) of Definition~\ref{DCus}. Moreover, $L_\rho$ also satisfies the property (vi) of Theorem~\ref{THeadIrr} if $\height(n\rho)\leq\height(\al)$.

\item[{\rm (2)}] The proper standard modules $\bar\De(\pi)$ for all $\pi\in\Pi(\al)$, defined as in (\ref{EStand}) using the modules from (1), satisfy the properties (i)--(v) of Theorem~\ref{THeadIrr}. 
\end{enumerate}

The induction starts with $\height(\al)=0$, and for $\height(\al)=1$ the theorem is also clear since $R_{\al_i}$ is a polynomial algebra, which has only the trivial irreducible (graded) representation $L_{\al_i}$. The inductive assumption will stay valid throughout \S~\ref{SRough}.

\subsubsection{} 
In the following proposition, we exclude the case where the proper standard module is of the form $L_\rho^{\circ n}$. The excluded cases will be dealt with in  \S\S~\ref{SSCusp} and \ref{SSPower}.

\begin{Proposition} \label{PHeadIrr} 
Let $\pi=(m_1,\dots,m_N)\in \Pi(\al)$, 
and suppose that 
there are $1\leq k\neq l\leq N$ such that $m_k\neq 0$ and $m_l\neq 0$. 
\begin{enumerate}
\item[{\rm (i)}] 
$
\Stand(\pi)
$ has an irreducible head; denote this irreducible module $L(\pi)$. 

\item[{\rm (ii)}] If $\pi\neq \si$, then $L(\pi)\not\simeq L(\si)$. 

\item[{\rm (iii)}] $L(\pi)^\circledast\cong L(\pi)$.  

\item[{\rm (iv)}] $[\Stand(\pi):L(\pi)]_q=1$, and $[\Stand(\pi):L(\si)]_q\neq 0$ implies $\si\leq \pi$. 

\item[{\rm (v)}] $\Res_{|\pi|}L(\pi)\simeq L_{\pi}$ and $\Res_{|\si|}L(\pi)\neq 0$ implies $\si\leq \pi$.  
\end{enumerate}
\end{Proposition}

\begin{proof}
(i) and (v) If $L$ is an irreducible quotient of $\bar\De(\pi)=\Ind_{|\pi|}L_{\pi}$, then by adjointness of $\Ind_{|\pi|}$ and $\Res_{|\pi|}$ and the irreducibility of the $R_{|\pi|}$-module $L_{\pi}$, which holds by the inductive assumption, we conclude  
that  $L_{\pi}$ is a submodule of $\Res_{|\pi|} L$. 
On the other hand, by Proposition~\ref{P1}(i) the multiplicity of $L_{\pi}$ in $\Res_{|\pi|} \bar\De(\pi)$ is $1$, so (i) follows. Note that we have also proved the first statement in (v), while the second statement in (v) follows from Proposition~\ref{P1}(ii) and the exactness of the functor $\Res_{|\pi|}$. 

(iv) By (v), $\Res_{|\si|}L(\si)\cong L_{\si}\neq 0$. Therefore, if $L(\si)$ is a composition factor of $\Stand(\pi)$, then $\Res_{|\si|}\Stand(\pi)\neq 0$ by exactness of $\Res_{|\si|}$. By Proposition~\ref{P1}, we then have $\si\leq \pi$ and (iv). 

(ii) If $L(\pi)\simeq L(\si)$, then we deduce from (iv) that $\pi\leq \si$ and $\si\leq \pi$, whence $\pi=\si$. 

(iii) follows from (v) and Lemma~\ref{LLMMuSeldDual}. 
\end{proof}

\subsubsection{}\label{SSCusp}
We now assume that $\al=\rho_k\in\Phi_+$.
There is a {\em trivial} root partition $(\rho_k)\in\Pi(\al)$. 
Proposition~\ref{PHeadIrr} yields $|\Pi(\al)|-1$ irreducible $R_\al$-modules, namely the ones which correspond to the {\em non-trivial}\, root partitions $\pi\in\Pi(\al)$. We define the cuspidal module $L_\al$ to be the missing irreducible $R_\al$-module, cf. Lemma~\ref{LAmount}. Then, of course, we have that $\{L(\pi)\mid\pi\in\Pi(\al)\}$ is a complete and irredundant system of irreducible $R_\al$-modules up to isomorphism. We now prove that $L_\al$ satisfies the property (Cus) and is uniquely determined by it:

\begin{Lemma} \label{LMcNamara}
Let $\al=\rho_k\in\Phi_+$. If  $\be,\ga\in Q_+$ are non-zero elements such that $\al=\be+\ga$ and $\Res_{\be,\ga}L_\al\neq 0$, then $\beta$ is a sum of roots less than $\al$ and $\ga$ is a sum of roots greater than $\al$. Moreover, this property characterizes $L_\al$ among the irreducible $R_\al$-modules uniquely up to isomorphism and degree shift.   
\end{Lemma}
\begin{proof}
We prove that $\beta$ is a sum of roots less than $\al$, the proof that $\ga$ is a sum of roots greater than $\al$ being similar. Let $L(\pi)\boxtimes L(\si)$ be an irreducible submodule of $\Res_{\be,\ga} L_\al$, so that $\pi=(m_1,\dots,m_N)\in\Pi(\be)$ and $\si=(n_1,\dots,n_N)\in\Pi(\ga)$. Let $a$ be minimal with $m_a\neq 0$.  Then $\Res_{\rho_a,\be-\rho_a}L(\pi)\neq 0$, and hence $\Res_{\rho_a,\ga+\be-\rho_a}L_\al\neq0$. If we can prove that $\rho_a$ is a sum of roots less than $\al$, then by convexity, $\rho_a$ is a root less than $\al$,  whence, by the minimality of $a$, we have that $\beta$ is a sum of roots less than $\al$. 
So we may assume from the beginning that $\be$ is a root and $L(\pi)=L_\be$. Moreover, we may assume that $\be$ is the maximal positive root for which $\Res_{\be,\ga} L_\al\neq 0$. 

Now, let $l$ be the minimal with $n_l\neq 0$. Then we have a non-zero map 
$$L_\be\boxtimes L_{\rho_l}\boxtimes V\to \Res_{\be,\kappa,\ga-{\rho_l}}L_\al,$$ 
for some $0\neq V\in\mod{R_{\ga-{\rho_l}}}$. By adjunction, this yields a non-zero map
$$
f: (\Ind_{\be,{\rho_l}} L_\be\boxtimes L_{\rho_l})\boxtimes V\to \Res_{\be+{\rho_l},\ga-{\rho_l}}L_\al.
$$

If ${\rho_l}=\ga$, then we must have $\be\prec\ga$, for otherwise $L_\al$ is a quotient of the proper standard module $L_\be\circ L_\ga$, which contradicts the definition of the cuspidal module $L_\al$. Now, since $\al=\be+{\rho_l}$, we have by convexity that $\be\prec\al\prec\ga$, in particular $\be\prec\al$ as desired. 

Next, let ${\rho_l}\neq\ga$, and pick a composition factor  $L(\pi')$ of $\Ind_{\be,{\rho_l}} L_\be\boxtimes L_{\rho_l}$, which is not in the kernel of $f$. Write $\pi'=(m_1',\dots,m_N')\in \Pi(\be+\rho_l)$. By the assumption on the maximality of $\beta$, we have $\rho_c\preceq \be$ whenever $m_c'>0$. 
Thus $\be+{\rho_l}$ is a sum of roots $\preceq \be$. Lemma~\ref{l1} implies that ${\rho_l}\preceq \be$, and so by adjointness, $L_\al$ is a quotient of the proper standard module $L_\beta\circ\bar\De(\si)$, which is a contradiction. 

The second statement of the lemma is clear since, in view of Proposition~\ref{PHeadIrr}(v) and Lemma~\ref{l1}, the irreducible modules $L(\pi)$, corresponding to non-trivial root partitions $\pi\in \Pi(\al)$, do not satisfy the property (Cus). 
\end{proof}

\subsubsection{}\label{SSPower}
Assume now that $\al=n\rho_k$ for some $\rho_k\in\Phi_+$ and $n\in\Z_{>1}$.

\begin{Lemma} \label{LCuspPower} 
The induced module $L_{\rho_k}^{\circ n}$ is irreducible. 
\end{Lemma}
\begin{proof}
In view of Proposition~\ref{PHeadIrr}, we have the irreducible modules $L(\pi)$ for all root partitions $\pi\in\Pi(\al)$, except for $\pi=\si:=(\rho_k^n)$, for which $\bar\De(\si)=L_{\rho_k}^{\circ n}$. By Lemma~\ref{l1}, $\si$ is the unique minimal element of $\Pi(\al)$. 
By Proposition~\ref{PHeadIrr}(v), we conclude that 
$L_{\rho_k}^{\circ n}$ has only one composition factor $L$ appearing with certain multiplicity $c(q)\in\A$, and such that $L\not\cong L(\pi)$ for all $\pi\in\Pi(\al)\setminus\{\si\}$. Finally, by Corollary~\ref{CPowerIrr}, we conclude that $L_{\rho_k}^{\circ n}\simeq L$. 
\end{proof}

The proof of Theorem~\ref{THeadIrr} is now complete.

\subsection{\boldmath Reduction modulo $p$}
In this subsection we work with two fields: $F$ of characteristic $p>0$ and $K$ of characteristic $0$. We use the corresponding indices to distinguish between the two situations. Given an irreducible $R_\al(K)$-module $L_K$ for a root partition $\pi\in\Pi(\al)$ we can pick a (graded) $R_\al(\Z)$-invariant lattice $L_\Z$ as follows: 
pick a homogeneous word vector $v\in L_K$ and set $L_\Z:=R_\al(\Z)v$. 
The lattice $L_\Z$ can be used to {\em reduce modulo $p$}:
$$
\bar L:=L_\Z\otimes_\Z F.
$$
In general, the $R_\al(F)$-module $\bar L$ depends on the choice of the lattice $L_\Z$. However, we have $\CH \bar L=\CH L_K$, so by linear independence of  characters of irreducible $R_\al(F)$-modules, composition multiplicities of irreducible $R_\al(F)$-modules in $\bar L$ are well-defined. In particular, we have  well-defined {\em decomposition numbers}
$$
d_{\pi,\si}:=[\bar L(\pi):L_F(\si)]_q\qquad (\pi,\si\in \Pi(\al)),
$$
which depend only on the characteristic $p$ of $F$, since prime fields are splitting fields for irreducible modules over KLR algebras.

\begin{Lemma} \label{LMultOneRed}
Let $L_K$ be an irreducible $R_\al(K)$-module and let $\bi=i_1^{a_1}\dots i_b^{a_b}$ be an extremal word for $L_K$. Let $N$ be the irreducible $\circledast$-selfdual $R_\al(F)$-module defined by
$
N:=\tilde f_{i_k}^{a_k}\dots\tilde f_{i_1}^{a_1}1_F.
$ 
Then $[\bar L:N]_q=1$. 
\end{Lemma}
\begin{proof}
Reduction modulo $p$ preserves formal characters, so the result follows from Corollary~\ref{CExtrNew}. 
\end{proof}

\begin{Proposition} \label{PRedModP}
Let $\pi,\si\in \Pi(\al)$. Then $d_{\pi,\si}\neq 0$ implies $\si\leq \pi$. In particular, reduction modulo $p$ of any cuspidal module is an irreducible cuspidal module again: $\bar L_{\rho}\simeq L_{\rho,F}$. 
\end{Proposition}
\begin{proof}
By Theorem~\ref{THeadIrr}(v), which holds over any field, we conclude that any composition factor of $\bar L_{\rho}$ is isomorphic to $L_{\rho,F}$ up to a degree shift. Now use Lemma~\ref{LMultOneRed}. 
\end{proof}

\subsection{PBW bases and canonical bases}
We now return to the algebra $\f$ and recall some results on its  PBW bases. For a fixed convex order on $\Phi_+$, 
Lusztig used a certain braid group action to define {\em root vectors}
$\{r_\rho\:|\:\rho \in \Phi_+\}$ in $\f$. The corresponding {\em dual root vectors}  
\begin{equation}\label{drv}
r_\rho^* := (1-q_\rho^{2}) r_\rho \qquad(\rho\in\Phi_+)
\end{equation}
are invariant under $\barinv^*$. 

For $\pi = (m_1,\dots,m_N) \in \Pi(\al)$,
we set
\begin{equation}\label{rlambda}
r_\pi := \frac{r_{\rho_1}^{m_1}}{[m_1]_{\rho_1}^!} \dots \frac{r_{\rho_N}^{m_N}}{[m_N]_{\rho_N}^!},
\qquad
r_\pi^* := q^{\shift(\pi)} (r_{\rho_1}^*)^{m_1} \cdots (r_{\rho_N}^*)^{m_N}.
\end{equation}

\begin{Theorem}\cite{Lubook} \label{pbw}
Let $\al\in Q_+$. Then  
$\left\{r_{\pi}\:|\:\pi \in \Pi(\al)\right\}$
and 
$\left\{r^*_{\pi}\:|\:\pi \in \Pi(\al)\right\}$
 are a pair of dual bases for the free $\A$-modules 
$(\f_{\A})_\al$ and $(\f_{\A}^*)_\al$ respectively.
\end{Theorem}

One can use the
$\barinv^*$-invariance of the dual root vectors together with
the Levendorskii-Soibelman formula \cite[Proposition 5.5.2]{LS} or  \cite[Proposition 1.9]{Lubraid} to deduce:
\begin{equation}\label{tr1}
\barinv^*(r_\pi^*) = r_\pi^* + (\text{a $\Z[q,q^{-1}]$-linear
  combination of $r_\si^*$ for $\si < \pi$}).
\end{equation}
\begin{equation}\label{tr2}
\barinv(r_\pi) = r_\pi + (\text{a $\Z[q,q^{-1}]$-linear
  combination of $r_\si$ for $\si > \pi$}).
\end{equation}
In view (\ref{tr1})--(\ref{tr2}) 
and Lusztig's Lemma, there exist unique bases
$\{b_\pi\:|\:\pi \in \Pi(\al)\}$ and $\{b_\pi^*\:|\:\pi \in
\Pi(\al)\}$
for $(\f_{\A})_\al$ and $(\f_{\A}^*)_\al$, respectively, such that
\begin{align}
\barinv(b_\pi) &= b_\pi,
 &b_\pi &= r_\pi + \text{(a $q \Z[q]$-linear combination of
  $r_\si$ for $\si > \pi$)},\\
\barinv^*(b^*_\pi) &= b^*_\pi,
&b_\pi^*  &= r^*_\pi
 + \text{(a $q \Z[q]$-linear combination of
  $r^*_\si$ for $\si < \pi$).}\label{tri}
\end{align}
These are the canonical and dual canonical bases, respectively (cf. 
\cite{Lu}
in simply-laced types
or \cite{Saito} in non-simply-laced types).

\subsection{Cuspidal modules and dual PBW bases}\label{SSCMDRE}
We continue to work with a fixed convex order $\prec $  on $\Phi_+$. 
Suppose that we are given elements
\begin{equation}\label{ERootElements}
\{E_\rho^*\in(\f_\A^*)_\rho\mid\rho\in\Phi_+\}.
\end{equation}
If $\pi=(m_1,\dots,m_N)$ is a root partition, define the corresponding {\em dual PBW monomial}
$$
E^*_\pi:=q^{\shift(\pi)}(E_{\rho_1}^*)^{m_1}\dots(E_{\rho_N}^*)^{m_N} \in\f_\A^*.
$$
We say that (\ref{ERootElements}) is a {\em dual PBW family}\,  if the following properties are satisfied: 
\begin{enumerate}
\item[{\rm (i)}] (`convexity') if $\be\succ\ga$ are positive roots then 
$E_\ga^* E_\be^*-q^{-(\be,\ga)}E_\be^* E_\ga^*$ is an $\A$-linear combination of elements $E^*_\pi$ with $\pi<(\be,\ga)\in\Pi(\be+\ga)$;

\item[{\rm (ii)}] (`basis') $\{E^*_\pi\mid\pi\in\Pi(\al)\}$ is an $\A$-basis of $(\f^*_\A)_\al$ for all $\al\in Q_+$; 

\item[{\rm (iii)}] (`orthogonality') $$(E^*_{\pi},E^*_{\si})=\de_{\pi,\si}\prod_{k=1}^N((E_{\rho_k}^*)^{m_k},(E_{\rho_k}^*)^{m_k});$$

\item[{\rm (iv)}] (`bar-triangularity') $\barinv^*(E^*_{\pi}) = E^*_{\pi} +$ an $\A$-linear combination of dual PBW monomials $E^*_{\si}$ for $\si < \pi$.


\end{enumerate}

The following result shows in particular that the elements $E_\rho^*$ of the dual PBW family are determined uniquely up to signs (for a fixed preorder $\preceq$):

\begin{Lemma} \label{LREUnique}
Assume that (\ref{ERootElements}) is a dual PBW family. Then:
\begin{enumerate}
\item[{\rm (i)}] The elements of (\ref{ERootElements}) are $\barinv^*$-invariant.
\item[{\rm (ii)}] Suppose that we are given another family $\{{}'E_\rho^*\in(\f_\A^*)_\rho\mid\rho\in\Phi_+\}$ of $\barinv^*$-invariant elements which satisfies the basis and orthogonality properties. 
Then $E_\rho^*=\pm\,{}'E_\rho^*$ for all $\rho\in\Phi_+^\re$.\end{enumerate}
\end{Lemma}
\begin{proof}
(i) The convexity of $\prec$ implies that for $\rho\in\Phi_+$ the root partition $(\rho)\in \Pi(\rho)$ is a minimal element of $\Pi(\rho)$. So the bar-triangularity property (iv) implies that  the elements of a dual PBW family are $\barinv^*$-invariant. 

(ii) We apply induction on $\height(\rho)$, the induction base being clear. By the basis property of dual PBW families, we can write  
\begin{equation}\label{E081212}
{}'E_\rho^*=cE_\rho^*+\sum_{\pi\in\Pi(\rho)\setminus\{(\rho)\}}c_{\pi}E^*_{\pi}\qquad(c,c_{\pi}\in\A).
\end{equation}

Fix for a moment a root partition $\pi\in \Pi(\rho)\setminus\{(\rho)\}$. 
By the orthogonality property of dual PBW families and non-degeneracy of the form $(\cdot,\cdot)$, 
the element $X_{\pi}:=\frac{1}{(E^*_\pi,E^*_\pi)} E^*_{\pi}$ satisfies $(E^*_{\si},X_{\pi})=\de_{\si,\pi}$ for all $\si\in\Pi(\rho)$. So pairing the right hand side of (\ref{E081212}) with $X_{\pi}$ yields $c_{\pi}$. On the other hand, by the inductive assumption, $E^*_{\pi}=\pm {}'E^*_{\pi}$. So using the orthogonality property for the primed family in (ii), we must have $({}'E_\rho^*,X_{\pi})=0$ for all  $\pi\in\Pi(\rho)\setminus\{(\rho)\}$. So $c_{\pi}=0$. 
Thus ${}'E_\rho^*=cE_\rho^*$. Furthermore, the elements ${}'E_\rho^*$ and $E_\rho^*$ belong to the algebra $\f_\A^*$ and are parts of its $\A$-bases, whence  ${}'E_\rho^*=\pm q^n E_\rho^*$. Since both ${}'E_\rho^*$ and $E_\rho^*$ are $\barinv^*$-invariant, we conclude that $n=0$. 
\end{proof}

\begin{Proposition}\label{PPBWFamily} 
The following set of elements in $\f_\A^*$
$$
\{E_\rho^*:=\ga^*([L_\rho])\mid\rho\in\Phi_+\}
$$
is a dual PBW family. 
\end{Proposition}
\begin{proof}
Under the categorification map $\ga^*$, the graded duality $\circledast$ corresponds to $\barinv^*$, so $\ga^*([L])$ is $\barinv^*$-invariant for any $\circledast$-self-dual $R_\al$-module $L$. 
Moreover, under $\ga^*$, the induction product corresponds to the product in $\f_\A^*$, so the convexity condition (i) follows from Theorem~\ref{THeadIrr}(iv) and Lemma~\ref{LDualInd}. Now, note that $E^*_{\pi}=\ga^*([\bar\De(\pi)])$, so the conditions (ii) and (iv) follow from Theorem~\ref{THeadIrr}(iv) again. It remains to establish the orthogonality property (iii). Let $\pi=(m_1,\dots,m_N)$. Under $\ga^*$, the coproduct $r$ corresponds to the map on the Grothendieck group induces by $\Res$. So using (\ref{EDefPropForm}), we get 
$$
(E^*_{\pi},E^*_{\si})=\big((E_{\rho_1}^*)^{m_1}
\otimes \dots  \otimes (E_{\rho_{N}}^*)^{m_{N}}, \ga^*([\Res_{|\pi|}\bar\De(\si)])\big).
$$
By Proposition~\ref{P1}, $\Res_{|\pi|}\bar\De(\si)=0$ unless $\pi\leq\si$, and for $\pi=\si$ we have 
$$\Res_{|\pi|}\bar\De(\si)=L_{\rho_1}^{\circ m_1}\boxtimes\dots \boxtimes L_{\rho_{N}}^{\circ m_{N}}.$$ Since the form $(\cdot,\cdot)$  is symmetric, the orthogonality  follows from the preceding remarks.

\end{proof}

It is shown in Lusztig~\cite{Lubook} and  \cite{Saito} that $\{r_\rho^*\mid \rho\in \Phi_+\}$ is a dual PBW family. Since the dual PBW families are unique up to a sign by Lemma~\ref{LREUnique}, it follows that $\ga^*([L_\rho])=\pm r^*_\rho$ for all $\rho\in\Phi_+$. In fact:

\begin{Proposition}\label{PDualPBWCat}
For every $\rho\in\Phi_+$ we have that $\ga^*([L_\rho])= r^*_\rho=b_{(\rho)}^*$ is a dual canonical basis element.
\end{Proposition}
\begin{proof}
By (\ref{tri}), we have $r_\rho^*=b_{(\rho)}^*$ is a dual canonical basis element. Now, in view of the commutativity of the triangle (\ref{ETriangle}), to show that $E_\rho^*= b_{(\rho)}^*$, it suffices to know that for an arbitrary element $b^*$ of the dual canonical basis, there exists at least one word $\bi\in\words$ such that the coefficient of $\bi$ in $\iota(b^*)$ evaluated at $q=1$ is positive. But this follows from Lemma~\ref{LExtrDCB}.
\end{proof}

\section{Homological properties of KLR algebras}\label{SHPKLRA} 
We now review some `standard homological properties' of KLR algebras of finite Lie type. We continue to work with a fixed convex order $\prec$ on $\Phi_+$. We mainly follow \cite{BKM}, to where we refer the reader for detailed proofs. 

\subsection{Finiteness of global dimension}
First of all, we record a key fundamental fact:

\begin{Theorem} 
If the Cartan matrix\, $\Car$ is of finite type, the KLR algebra $R_\al(\Car)$ has  global dimension equal to $\height(\al)$ (as a graded algebra). 
\end{Theorem}

The finiteness of the global dimension of $R_\al(\Car)$ (as a graded algebra) was first proved by Kato \cite{Kato} for the case $\cha F=0$ and $\Car$ of finite ${\tt ADE}$ types. For an arbitrary $F$ and $\Car$ of finite ${\tt BCFG}$ types McNamara \cite{McN} computed the global dimension explicitly as $\height(\al)$. In \cite[Appendix]{BKM}, it was verified that the methods of \cite{McN} also lead to the same answer for  finite ${\tt ADE}$ types over any field, not surprisingly the case $E_8$ being the most difficult. 

Still for $\Car$ of finite type, the algebras $R_\al(\O)$ are affine quasi-hereditary. This is shown in \cite{KLM} for finite type ${\tt A}$ and in \cite{KLoub} for other finite Lie types. From this we have the following slight generalization: if $\O$ is a commutative ring of finite global dimension, then  $R_\al(\O,\Car)$  also has finite global dimension, even as an ungraded algebra. 
 
Finally, it can be checked that for a fixed $\Car$, the algebras $R_\al(F,\Car)$ have finite global dimension for all $\al\in Q_+$ if and only if  $\Car$ is of finite type. 


\subsection{Standard modules}
Throughout this subsection, $\rho\in \Phi_+$ is a fixed positive root.  
Recall the cuspidal module $L_\rho$. The proof of the following result relies on the finiteness of the global dimension of $R_\al$. 

\begin{Theorem}\label{source}{\rm \cite[\S4]{McN}}
Let $d \geq 1$. Then 
$$
\DIM\EXT^d_{R_\rho}(L_\rho, L_\rho) =
\left\{
\begin{array}{ll}
q_\rho^{2} &\hbox{if $d=1;$}\\
0 &\hbox{if $d \geq 2$.}
\end{array}
\right.
$$
\end{Theorem}

This theorem allows one to extend $L_\rho$ by $q_\rho^2L_\rho$, then by $q_\rho^4L_\rho$, etc. to get in the limit the modules $\Delta(\rho)$ with the following properties:

\begin{Theorem}\label{rm}{\rm \cite[Theorem 3.4, Corollary 3.4]{BKM}}
There is
a short exact sequence
\begin{equation}\label{ses3}
0 \longrightarrow q_\rho ^2 \Delta(\rho )
{\longrightarrow}
\Delta(\rho ) \longrightarrow L_\rho \longrightarrow 0.
\end{equation}
Moreover:
\begin{itemize}
\item[(i)]
$\Delta(\rho )$ is a cyclic module, and in the Grothendieck group we have 
\begin{equation}\label{EDeGrGr}
[\Delta(\rho )] = \frac{1}{1-q_\rho^2}\,[L_\rho ] ;
\end{equation}
\item[(ii)] $\Delta(\rho)$ has irreducible head isomorphic to
  $L_\rho$;
\item[(iii)] 
we have that 
$\EXT^d_{R_\rho}(\Delta(\rho), V) =0$ 
for $d \geq 1$ and
any finitely generated $R_\rho$-module $V$ with
  all composition factors $\simeq L_\rho$;
\item[(iv)]
$\END_{R_\rho}(\Delta(\rho)) \cong F[x]$ for $x$ in
degree $2d_\rho$.

\item[(v)]
The functor 
$\HOM_{R_\rho}(\Delta(\rho), -)$ defines an equivalence
from
the category of finitely generated graded
$R_\rho$-modules with all composition factors $\simeq L_\rho$ to the category of finitely generated
graded $\k[x]$-modules (viewing $\k[x]$ as a graded algebra with $\deg(x) =2 d_\rho$).
\end{itemize}
\end{Theorem}

\begin{Remark} \label{RR'}
{\rm 
In ${\tt ADE}$ types, there is a more elementary construction of $\De(\rho)$. For any $\al\in Q_+$ of height $d$, let  
$R_\al'$ be the subalgebra of
$R_\alpha$
generated by 
$$\{1_\bi\:|\:\bi \in \W_\alpha\}\cup \{y_1 -
y_2,\dots,y_{d-1}-y_d\}
\cup \{\psi_1,\dots,\psi_{d-1}\}.
$$
Denote by $L'_\rho$ the restriction of $L_\rho$ from $R_\rho$ to $R_\rho'$. Then $\De(\rho)\cong R_\rho \otimes_{R'_\rho} L'_\rho$.
}
\end{Remark}

By (\ref{drv}) and (\ref{EDeGrGr}), the module  
$\Delta(\rho )$ categorifies the root vector $r_\rho $. Compare this to  Proposition~\ref{PDualPBWCat}, which shows that $\bar\De(\rho)=L_\rho$ categorifies the dual root vector $r_\rho^*$. Next, we explain how to category the divided powers $r_\rho^m/[m]_\rho^!$ for all $m\in\Z_{\geq 0}$. 
For this we need to compute the endomorphism algebra of
$\Delta(\rho)^{\circ m}$.

Choose a non-zero homogeneous vector $v_\rho$ of minimal degree in $\Delta(\rho)$. It generates $\De(\rho)$ as an $R_\rho$-module. The proof of the following lemma is based on the Mackey Theorem and splitting coming from Theorem~\ref{rm}(iii). 

\begin{Lemma}\label{dpl1} {\rm \cite[Lemma 3.6]{BKM}}
Let $w \in \Si_{2n}$ be the permutation 
mapping $(1,\dots,n, n+1,\dots,2n)$ to $(n+1,\dots,2n,1,\dots,n)$.
There is a unique 
$R_{2\rho}$-module homomorphism
$$
\tau:
\Delta(\rho) \circ \Delta(\rho) \rightarrow \Delta(\rho) \circ
\Delta(\rho)
$$
of degree $-2d_\rho$
such that $\tau(1_{\rho,\rho} \otimes (v_\rho \otimes v_\rho))
= \psi_w 1_{\rho,\rho} \otimes (v_\rho \otimes v_\rho)$.
\end{Lemma}

Now pick a non-zero endomorphism $x
\in \END_{R_\rho}(\Delta(\rho))_{2d_\rho}$.
By Theorem~\ref{rm}(iv) we have that
$\END_{R_\rho}(\Delta(\rho)) = \k[x]$,
so $x$ is unique up to a scalar. Now we have commuting 
endomorphisms $\dx_1,\dots,\dx_m
\in \END_{R_{m\rho}}(\Delta(\rho)^{\circ m})_{2d_\rho}$
with
$$\dx_r := \id^{\circ(r-1)} \circ x \circ \id^{\circ (m-r)}.$$
Moreover, the endomorphism $\tau$ from the previous  lemma yields $\dtau_1,\dots,\dtau_{m-1} \in \END_{R_{m
    \rho}}(\Delta(\rho)^{\circ m})_{-2d_\rho}$
with 
$$\dtau_r := \id^{\circ (r-1)} \circ\tau\circ \id^{\circ(m-r-1)}.$$
Now \cite[Lemmas 3.7--3.9]{BKM} yield:

\begin{Theorem} 
For a unique choice of $x
\in \END_{R_\rho}(\Delta(\rho))_{2d_\rho}$, there is an 
algebra isomorphism
\begin{equation*}
\NH_m \stackrel{\sim}{\rightarrow}
\END_{R_{m\rho}}(\Delta(\rho)^{\circ m})^{\op},\qquad
y_i \mapsto \dx_i,
\ \psi_j \mapsto \dtau_j.
\end{equation*}
\end{Theorem}

By the theorem, we can view $\Delta(\rho)^{\circ
  m}$ as an $(R_{m\rho},
\NH_m)$-bimodule. Finally define
the {\em divided power module}
\begin{equation}\label{divided}
\Delta(\rho^m) := q_\rho^{m(m-1)/2}
\Delta(\rho)^{\circ m}e_m
\end{equation}
where $e_m \in \NH_m$ is the idempotent (\ref{idemp}).

\begin{Lemma}\label{dp} {\rm \cite[Lemmas 3.7--3.9]{BKM}} 
We have that $\Delta(\rho)^{\circ m} \cong [m]_\rho^! \Delta(\rho^m)$
as an $R_{m\rho}$-module.
Moreover $\Delta(\rho^m)$ has irreducible head
$L(\rho^m)$, and in the Grothendieck group we have 
$$
[\Delta(\rho^m)] = \frac{1}{(1-q_\rho^2)(1-q_\rho^4)\cdots (1-q_\rho^{2m})}\,[L(\rho^m)] .
$$
\end{Lemma}

The lemma shows that in
the Grothendieck group $[\Delta(\rho^m)]$ corresponds the  to the divided power $r_\rho^{m} / [m]_\rho^!$ under the Khovanov-Lauda-Rouquier categorification.
More generally, for a root partition $\pi=(m_1,\dots,m_N)$, define the {\em
  standard module}
\begin{equation}\label{standard}
\Delta(\pi) := \Delta(\rho_1^{m_1}) \circ\cdots \circ
\Delta(\rho_N^{m_N}).
\end{equation}

\begin{Theorem}\label{shead} {\rm \cite[Theorem 3.11]{BKM}} 
For a root partition $\pi=(m_1,\dots,m_N)$, the module $V_0 := \Delta(\pi)$ has an exhaustive
filtration
$V_0 \supset V_1 \supset
V_2 \supset \cdots$
such that $V_0 / V_1 \cong \bar\Delta(\pi)$ 
and all other sections 
of the form $q^{2m}\bar\Delta(\pi)$ for $m > 0$.
Moreover, $\Delta(\pi)$
has irreducible head $\cong L(\pi)$, and in the Grothendieck group: 
$$
[\Delta(\pi)] = 
[\bar \Delta(\pi)] \:\:\Big / \prod_{k=1}^N\prod_{r=1}^{m_k}(1-q_{\rho_k}^{2r}).
$$
\end{Theorem}

\subsection{Homological properties of standard modules} 
Now that we have constructed the standard modules, we list some of their homological properties. Throughout the subsection,  $\al\in Q_+$ is fixed. 

\begin{Theorem}\label{shp} {\rm \cite[Theorem 4.12]{Kato}, \cite[Theorm 3.12]{BKM}} 
Let $\pi  \in \Pi(\alpha)$. 
\begin{itemize}
\item[(i)] If 
$\EXT^d_{R_\alpha}(\Delta(\pi),V) \neq 0$
for some $d \geq 1$ and a finitely generated $R_\alpha$-module $V$, then $V$ has a composition factor $\simeq L(\si)$ for $\si\succ \pi$.
\item[(ii)] We have for all $d \geq 0$ and $\si\in\Pi(\al)$:
$$
\DIM \EXT^d_{R_\alpha}(\Delta(\pi), \bar\nabla(\si)) = \de_{d,0}\de_{\pi,\si}.
$$
\end{itemize}
\end{Theorem}

We say that an $R_\alpha$-module $V$ has a {\em $\Delta$-filtration}, written $V\in\Fil(\De)$, 
if there is a {\em finite} filtration $V = V_0 \supset V_1 \supset \cdots \supset V_n = 0$
such that $V_{i} / V_{i+1} \simeq \Delta(\pi^{(i)})$ for each
$i=1,\dots,n-1$ and
some $\pi^{(i)} \in \Pi(\alpha)$.
If $V\in\Fil(\De)$, then by Theorem~\ref{shp}(ii), the (graded) multiplicity of $\De(\pi)$ in a $\De$-filtration of $V$ is well-defined (i.e. independent of the $\De$-filtration) and is equal to that
\begin{equation*}
[V:\Delta(\pi)]_q = 
\overline{\DIM \HOM_{R_\alpha}(V, \bar\nabla(\pi))}\qquad(\pi\in\Pi(\al)).
\end{equation*}

\begin{Theorem}\label{wef} {\rm \cite[Theorem 3.13]{BKM}} 
Let $V$ be a finitely generated $R_\alpha$-module. Then $V\in\Fil(\De)$ if and only if $\EXT^1_{R_\alpha}(V, \bar\nabla(\si)) = 0$ for all 
$\si \in \Pi(\alpha)$.
\end{Theorem}

An immediate corollary is the following version of the `BGG reciprocity.' Note that the projective cover $P(\pi)$ of the irreducible module $L(\pi)$ exists in view of the general theory described in \S\ref{SSGeneral}. 

\begin{Corollary}\label{bgg} {\rm \cite[Remark 4.17]{Kato}, \cite[Corollary 4.17]{BKM}} 
Let $\pi,\si \in \Pi(\alpha)$. 
Then $P(\pi)\in\Fil(\De)$ and $[P(\pi):\Delta(\si)]_q = [\bar\Delta(\si):L(\pi)]_q$.
\end{Corollary}

This implies the following important dimension formula. A more elementary proof of this formula is given in \cite{KLoub}.

\begin{Corollary}{\rm \cite[Corollary 3.15]{BKM}}  
We have that
\begin{align*}
\DIM R_\alpha &= \!\!\sum_{\pi \in \Pi(\alpha)}
(\DIM \Delta(\pi)) (\DIM \bar\Delta(\pi))
\\
&= 
\sum_{\pi=(m_1,\dots,m_N) \in \Pi(\alpha)}
\frac{
(\DIM \bar\Delta(\pi))^2}
{\prod_{k=1}^N\prod_{r=1}^{m_k}(1-q_{\rho_k}^{2r})}.
\end{align*}
\end{Corollary}

The following corollary yields a  description of the standard modules $\Delta(\pi)$ and
$\bar\Delta(\pi)$ in spirit of standardly stratified algebras, cf. \cite[Corollary 4.18]{Kato}.

\begin{Corollary}\label{c} {\rm \cite[Corollary 3.15]{BKM}} 
Let $\pi \in \Pi(\al)$, and
$$
K(\pi):=\sum_{\si\not\preceq\pi}\,
\sum_{f\in\HOM_{R_\al}(P(\si),P(\pi))} \im f,\quad
\bar K(\pi):=\sum_{\si\not\prec\pi}\,
\sum_{f\in\HOM_{R_\al}(P(\si),P(\pi))} \im f.
$$
Then $\Delta(\pi) \cong
P(\pi)/K(\pi)$ and $\bar\De(\pi)\cong P(\pi)/\bar K(\pi)$. 
\end{Corollary}

\section{Projective resolutions of standard modules}\label{SPRSM}
We now explain how the standard modules $\Delta(\rho)$ 
fit into some short exact sequences, giving an alternative way to deduce their properties. This bounds the projective dimension of standard modules, and allows us to construct
some projective resolutions of standard modules. 
As usual, we work with a fixed convex order $\prec$ on $\Phi_+$ and denote by $\rho$ an arbitrary positive root.

\subsection{Minimal pairs}\label{minp}
A two-term root partition $\pi = (\beta,\gamma)\in \Pi(\rho)$ is called a {\em minimal pair} for $\rho$, if it is a minimal element of $\Pi(\rho)\setminus\{(\rho)\}$. 
Equivalently, a minimal pair for $\rho$ is a pair $(\beta,\gamma)$ 
of positive roots with $\beta+\gamma=\rho$ and
$\beta \succ \gamma$
such that there exists no other pair $(\beta',\gamma')$ of positive
roots with $\beta'+\gamma'=\rho$ and $\beta \succ \beta' \succ
\rho \succ \gamma' \succ \gamma$.
Let $\MP(\rho)$ denote the set of all minimal pairs for $\rho$.

For $\pi = (\beta,\gamma) \in \MP(\rho)$,
it follows from Theorem~\ref{THeadIrr} and the minimality of
$\pi$ that all composition 
factors of $\rad\, \bar\Delta(\pi)$ are $\simeq L_\rho$.  
Since $\bar\Delta(\pi) = L_\beta \circ L_\gamma $
and $(L_\beta \circ L_\gamma )^\circledast \cong q^{(\beta,\gamma)}L_\gamma  \circ
L_\beta $
by
Lemma~\ref{LDualInd}, there are short exact sequences
\begin{align}\label{sesonea}
0 \longrightarrow q^{-(\be,\ga)}\X^\circledast &\longrightarrow L_\beta \circ L_\gamma 
\longrightarrow L(\pi) \longrightarrow 0,\\
0 \longrightarrow q^{-(\be,\ga)} L(\pi)
&\longrightarrow L_\gamma \circ L_\beta 
\longrightarrow \X
\longrightarrow 0,\label{sestwoa}
\end{align}
where $\X := q^{-(\be,\ga)}(\rad\,
\bar\Delta(\pi))^\circledast$
is a finite dimensional module
with all composition factors $\simeq L_\rho $. 
It turns out that one can be much more precise. Let $\Phi=\Phi_+\sqcup -\Phi_+$ be the set of all roots. For any $\beta,\gamma \in \Phi$, define the number 
$$
p_{\beta,\gamma} := \max\{m \in \Z\:|\:\beta - m \gamma \in
  \Phi\}.
$$

\begin{Theorem}\label{as} {\rm \cite[Theorem 4.7, Corollary 4.3]{BKM}}
Let $\pi = (\beta,\gamma) \in \MP(\rho)$. 
\begin{enumerate}
\item[{\rm (i)}] There are short exact sequences
\begin{align}\label{sesoneas}
0 \longrightarrow q^{p_{\beta,\gamma}-(\be,\ga)}L_\rho  &\longrightarrow L_\beta \circ L_\gamma 
\longrightarrow L(\pi) \longrightarrow 0,\\
0 \longrightarrow q^{-(\be,\ga)} L(\pi)
&\longrightarrow L_\gamma \circ L_\beta 
\longrightarrow q^{-p_{\beta,\gamma}} L_\rho 
\longrightarrow 0.\label{sestwoas}
\end{align}

\item[{\rm (ii)}] In the Grothendieck group we have that
$$\left[\Res^{\rho}_{\gamma,\beta} L_\rho \right] = 
[p_{\beta,\gamma}+1]_q\,\big[L_\gamma  \boxtimes L_\beta \big].$$
\end{enumerate}
Moreover, $\Res_{\gamma,\beta} L_\rho$ is uniserial with socle $q^{p_{\beta,\gamma}} L_\gamma\boxtimes L_\beta$. 
\end{Theorem}

We now explain how minimal pairs allow one to get a generalization of high weight theory which was first developed in \cite{KR2} for the so-called Lyndon convex orders. 
Fix an arbitrary minimal pair 
$\minp(\rho) \in \MP(\rho)$ for each non-simple positive root $\rho \in \Phi_+$. 
Dependent on this choice, we
recursively define a word $\bi_\rho \in \W_\rho$
and a bar-invariant Laurent polynomial $\kappa_\rho \in \A$ for all $\rho\in\Phi_+$ as follows.  For $i
\in I$ set $\bi_{\al_i} := i$ and $\kappa_{\alpha_i} := 1$; then for non-simple $\rho \in \Phi_+$ suppose that $(\beta,\gamma) = \minp(\rho)$ and set
\begin{equation}\label{issue}
\bi_\rho := \bi_\gamma \bi_\beta,\qquad
\kappa_\rho := [p_{\beta,\gamma}+1]_q\,\kappa_\beta \kappa_\gamma.
\end{equation}
For example, in simply-laced types
we have that $\kappa_\rho = 1$ for all $\rho \in \Phi_+$; 
this is also the case in non-simply-laced types for 
{\em multiplicity-free} positive roots, i.e. roots 
$\rho = \sum_{i \in I} c_i \alpha_i$ with $c_i \in \{0,1\}$ for all
$i$.
Finally for a root partition $\pi = (\beta_1\geq\dots\geq\be_l)=(\rho_1^{m_1},\dots,\rho_N^{m_N})$ let
\begin{equation}
\bi_\pi := \bi_{\beta_1} \cdots \bi_{\beta_l},
\qquad
\kappa_\pi := \prod_{k=1}^N [m_k]^!_{\rho_k}\kappa_{\rho_k}^{m_k}.
\end{equation}
The following lemma shows that the words $\bi_\pi$
distinguish irreducible modules, generalizing \cite[Theorem 7.2(ii)]{KR2}.

\begin{Lemma}\label{rtheory} {\rm \cite[Lemma~4.5]{BKM}}
Let $\al \in Q_+$ and 
$\pi, \si \in \Pi(\al)$. Then $\DIM  L(\pi)_{\bi_\si} = 0$ if $\si \not\preceq
\pi$, and $\DIM  L(\pi)_{\bi_\pi} = \kappa_\pi$.
\end{Lemma}

\subsection{Projective resolutions}\label{sses}
Let $\rho \in \Phi_+$ be a non-simple positive root,  $(\beta,\gamma)$ be a minimal pair for $\rho$, 
and  $m := \height(\gamma)$. Let $w \in \Si_{n}$ be the permutation 
$$(1,\dots,n) \mapsto (n-m+1,\dots,n,1,\dots,n-m),$$ 
so that
$\psi_w 1_{\gamma,\beta} = 1_{\beta,\gamma} \psi_w$.
It is proved in \cite[Lemma 4.9]{BKM} that there is a unique homogeneous homomorphism 
\begin{equation}\label{EPhi}
\phi:q^{-(\beta,\gamma)}\,\Delta(\beta)\circ\Delta(\gamma)
\rightarrow \Delta(\gamma) \circ \Delta(\beta)
\end{equation}
such that 
$\phi(1_{\beta,\gamma} \otimes (v_1 \otimes v_2))
= \psi_w 1_{\gamma,\beta} \otimes (v_2 \otimes v_1)$
for all $v_1 \in \Delta(\beta), v_2 \in \Delta(\gamma)$.

\begin{Theorem}\label{inj3} {\rm \cite[Theorem 4.10]{BKM}}
For $(\beta,\gamma) \in \MP(\rho)$
there is a short exact sequence
$$
0 \longrightarrow q^{-(\beta,\gamma)}\, \Delta(\beta)\circ\Delta(\gamma)
\stackrel{\phi}{\longrightarrow} \Delta(\gamma)\circ\Delta(\beta)
\longrightarrow [p_{\beta,\gamma}+1]_q\,\Delta(\rho) \longrightarrow 0.
$$
\end{Theorem}

Let us again fix a choice of minimal pairs $\minp(\rho)
\in \MP(\rho)$ for each $\rho \in \Phi_+$ of height at least two, and
recall $\kappa_\rho$ and $\kappa_\pi$ from (\ref{issue}).
Let
\begin{equation}
\tilde{\Delta}(\rho) := \kappa_\rho \Delta(\rho),
\qquad
\tilde{\Delta}(\pi) := \kappa_\pi \Delta(\pi).
\end{equation}

For simply laced $\Car$ we have $\tilde{\Delta}(\rho) =\Delta(\rho)$. 
We want to construct a 
projective resolution $\Resol(\rho)$ of $\tilde{\Delta}(\rho)$
for each $\rho \in \Phi_+$. 
Then more generally, given $\alpha \in Q^+$ and
$\pi = (\be_1\succeq\be_2 \succeq\dots\succeq \be_l) \in \Pi(\alpha)$,
the total complex of the `$\circ$'-product of the complexes
$\Resol(\be_1),\dots,\Resol(\be_l)$ gives a projective resolution
$\Resol(\pi)$
of $\tilde\Delta(\pi)$.

The resolution $\Resol(\rho)$ is going to be of the form
$$
\Resol(\rho):\qquad\qquad 0\to P_{n-1}(\rho)\longrightarrow \dots\longrightarrow P_1(\rho)\longrightarrow P_0(\rho)\longrightarrow \tilde\De(\rho)\to 0,
$$
where $n=\height(\rho)$. 
The construction of $\Resol(\rho)$ is recursive. For $i \in I$ we have $\tilde{\Delta}(\alpha_i) = R_{\alpha_i}$,
which is projective already. 
So we just set $P_0(\alpha_i) := R_{\alpha_i}$
and $P_d(\alpha_i) := 0$ for $d \neq 0$ to obtain the required
resolution.
Now suppose that $\rho \in \Phi_+$ is of height at least two
and let $(\beta,\gamma) := \minp(\rho)$, a fixed minimal pair for $\rho$.
We may assume by induction that the projective resolutions $\Resol(\beta)$
and $\Resol(\gamma)$ are already defined.
Taking the total complex of their `$\circ$'-product using \cite[Acyclic Assembly Lemma 2.7.3]{Wei}, 
we obtain a projective resolution
$\Resol(\beta,\gamma)$ of $\tilde{\Delta}(\beta) \circ
\tilde{\Delta}(\gamma)$
with
\begin{align*}
P_d(\beta,\gamma) &:= \bigoplus_{d_1+d_2= d} P_{d_1}(\beta) \circ
P_{d_2}(\gamma),\\
\qquad \partial_d &:= \left(\id \circ \partial_{d_2}
-(-1)^{d_2}\partial_{d_1} \circ \id
\right)_{d_1+d_2=d}:P_d(\beta,\gamma) \rightarrow P_{d-1}(\beta,\gamma).
\end{align*}
Similarly 
we 
obtain a projective resolution $\Resol(\gamma,\beta)$ of
$\tilde{\Delta}(\gamma)\circ\tilde{\Delta}(\beta)$ with
\begin{align*}
P_d(\gamma,\beta) &:= \bigoplus_{d_1+d_2= d} P_{d_1}(\gamma) \circ
P_{d_2}(\beta),\\
\qquad \partial_d &:= \left(\partial_{d_1} \circ \id + (-1)^{d_1}
\id \circ \partial_{d_2}\right)_{d_1+d_2=d}:P_d(\gamma,\beta) \rightarrow P_{d-1}(\gamma,\beta).
\end{align*}
There is an injective homomorphism
$$
\tilde{\phi}:
q^{-(\be,\ga)}
\tilde{\Delta}(\beta)\circ\tilde{\Delta}(\gamma)
\hookrightarrow \tilde{\Delta}(\gamma) \circ
\tilde{\Delta}(\beta)
$$
defined in exactly the same way as the map $\phi$ in (\ref{EPhi}),
indeed, it
is just a direct sum of copies of the map $\phi$ from there.
Applying \cite[Comparision Theorem 2.2.6]{Wei}, 
$\tilde{\phi}$ lifts to a 
chain map $\tilde{\phi}_*:q^{-(\be,\ga)} \Resol(\beta,\gamma)
\rightarrow \Resol(\gamma,\beta)$.
Then we take the mapping cone of $\tilde{\phi}_*$ to obtain a complex
$\Resol(\rho)$ with
\begin{align*}
\qquad P_d(\rho) &:= P_d(\gamma,\beta) \oplus q^{-(\be,\ga)} P_{d-1}(\beta,\gamma),\\
\partial_d &:= (\partial_d, \partial_{d-1}+(-1)^{d-1} \tilde{\phi}_{d-1}):
P_d(\rho) \rightarrow P_{d-1}(\rho).
\end{align*}
In view of Theorem~\ref{inj3} and \cite[Acyclic Assembly Lemma 2.7.3]{Wei}
once again, 
$\Resol(\rho)$ is a projective resolution of $\tilde{\Delta}(\rho)$.

Let us describe
$\Resol(\rho)$ more explicitly.
First, for $i \in I$ and the empty tuple $\bsi$,
set $\bi_{\alpha_i,\bsi} := i$.
Now suppose that $\rho$ is of height $n \geq 2$ and
that $(\beta,\gamma) = \minp(\rho)$ with $\gamma$ of height $m$.
For $\bsi  = (\sigma_1,\dots,\sigma_{n-1})
\in \{0,1\}^{n-1}$,
let $$|\bsi| := \sigma_1+\cdots+\sigma_{n-1},\ 
\bsi_{< m} := (\sigma_1,\dots,\sigma_{m-1}),\ 
\bsi_{> m} := (\sigma_{m+1},\dots,\sigma_{n-1}).
$$
Define $\bed_{\rho,\bsi} \in \W_\rho$ 
and $d_{\rho,\bsi} \in \Z_{\geq 0}$
recursively from
\begin{align*}
\bed_{\rho,\bsi} &:= \left\{
\begin{array}{ll}
\bed_{\gamma,\bsi_{< m}}
\bed_{\beta,\bsi_{> m}}\hspace{16.5mm}&\text{if $\sigma_m = 0$,}\\
\bed_{\beta,\bsi_{> m}}
\bed_{\gamma,\bsi_{< m}}
&\text{if $\sigma_m = 1$;}
\end{array}\right.\\
d_{\rho,\bsi} &:= \left\{
\begin{array}{ll}
d_{\beta,\bsi_{> m}}
+d_{\gamma,\bsi_{< m}}
&\text{if $\sigma_m = 0$,}\\
d_{\beta,\bsi_{> m}}+
d_{\gamma,\bsi_{< m}}-(\be,\ga)
&\text{if $\sigma_m = 1$.}
\end{array}\right.
\end{align*}
Note in particular that $d_{\rho,\bsi} = |\bsi|$ for simply-laced $\Car$.
Also if $\bsi = (0,\dots,0)$ then $\bed_{\rho,\bsi}$ is the tuple $\bed_\rho$ from
(\ref{issue}) and $d_{\rho,\bsi}=0$.
Then 
we have that
\begin{equation}\label{exp}
P_d(\rho) =\bigoplus_{\substack{\bsi \in \{0,1\}^{n-1}\\ 
|\bsi| = d}}
q^{d_{\rho,\bsi}} R_\rho 1_{\bed_{\rho,\bsi}}.
\end{equation}
For the differentials 
$\partial_d:P_d(\rho)\rightarrow P_{d-1}(\rho)$,
 there are elements $\psi_{\bsi,\btau} \in
 1_{\bed_{\rho,\bsi}} 
R_\rho 1_{\bed_{\rho,\rho}}$
for each $\bsi,\btau \in \{0,1\}^{n-1}$
with $|\bsi|=d, |\btau| = d-1$
such that, on
viewing elements of (\ref{exp}) as row vectors,
the differential $\partial_d$ is 
defined by right multiplication by the matrix
$\left(\psi_{\bsi,\btau}
\right)_{|\bsi|=d,|\btau|=d-1}$.
Moreover 
$\psi_{\bsi,\btau} = 0$ unless the tuples $\bsi$ and $\btau$ 
differ in just one entry.
We are able to give a very explicit description
of the elements $\psi_{\bsi,\btau}$ in the following special case.

\begin{Theorem}\label{expres}
Suppose that $\rho \in \Phi_+$ is {multiplicity-free}, so that
$\kappa_\rho=1$
and $\Resol(\rho)$
is a projective resolution of the root module $\Delta(\rho)$.
Then the elements $\psi_{\bsi,\btau}$
may be chosen
so that
$$
\psi_{\bsi,\btau} := (-1)^{\sigma_1+\cdots+\sigma_{r-1}} \psi_w
$$
if $\bsi$ and $\btau$ differ just in the $r$th entry,
where $w\in \Si_n$ is the {\em unique} permutation with
$1_{\bed_{\rho,\bsi}} \psi_w = \psi_w 1_{\bed_{\rho,\btau}}$.
\end{Theorem}

\section{Type ${\tt A}$}
In this section, we sketch an elementary approach to the homological theory described above in type $\Car={\tt A_\infty}$, which is equivalent to working in an arbitrary finite type ${\tt A}$. The proofs we give are independent of the theory described in \S\S\ref{SHPKLRA},\ref{SPRSM}.  

\subsection{Set up}
Throughout the section $\O$ is an arbitrary commutative unital ring. Sometimes we need to assume that  $\O=F$, i.e. that $\O$ is a field. 
We often work with a fixed arbitrary positive root of the form 
$$\rho=\rho(k,l):=\al_k+\al_{k+1}+\dots+\al_l\in\Phi_+\qquad(k\leq l).
$$ 
We then have $d:=\height(\rho)=l-k+1$. We refer to $R_\rho$ as a {\em cuspidal block}. We also set $n:=d-1$, and define the word
$$
\bi_\rho=(k,k+1,\dots,l).
$$
We fix the convex order $\prec$ with $\al_k\prec \al_l$ if and only if $k<l$. Then $\rho(k,l)\prec\rho(r,s)$ if and only if $\bi_{\rho(k,l)}<\bi_{\rho(r,s)}$ in the the usual lexicographic order. 

For $\rho=\rho(k,l)$, the corresponding {\em cuspidal module} $L_\rho$ is the rank one $\O$-module $\O\cdot v_\rho$ with the action of $R_\rho$ on the basis vector $v_\rho$ defined by 
$$
1_\bj v_\rho=\de_{\bj,\bi_{\rho}}v_\rho,\quad y_r v_\rho=0,\quad \psi_t v_\rho=0
$$
for all admissible $\bj,r,t$.  

For $\al\in Q_+$ recall the algebra $R_\al'$ defined in Remark~\ref{RR'}. 
Denote
$$
x_1=y_1-y_2,\ x_2=y_2-y_3,\dots, x_{d-1}=y_{d-1}-y_d,
$$
where $d=\height(\al)$. 
For $\bi\in \words_\al$, define the projective modules
$$
E(\bi):=R_\al 1_\bi\quad\text{and}\quad  E'(\bi)=R_\al'1_\bi
$$
over $R_\al$ and $R'_\al$ respectively. It is easy to see that 
\begin{equation}\label{EIndTwoProj}
E(\bi)\circ E(\bj)\cong E(\bi\bj).
\end{equation}

Pick any $\bi=(i_1,\dots,i_d)\in \words_\al$, and consider the degree $2$ element $z(\bi)\in R_\al$ defined as the sum of {\em distinct} basis elements of the form $y_{u\cdot 1}1_{u\cdot \bi}$ with $u\in\Si_d$. In other words, 
\begin{equation}\label{Ez}
z(\bi)=\sum_{\bj\in \words_\al}\Big(\sum_{1\leq r\leq d,\ j_r=i_1}y_r\Big)1_{\bj}.
\end{equation}
For example:
\begin{align*}
z(112)&=y_11_{112}+y_11_{121}+y_21_{112}+y_21_{211}+y_31_{121}+y_31_{211},
\\
z(211)&=y_11_{211}+y_21_{121}+y_31_{112}.
\end{align*}
Note by \cite[Theorem 2.9]{KL1} that $z(\bi)$ is central in $R_\al$. 

\begin{Proposition} \label{Pz} 
Let $\al=\sum_{i\in I}m_i\al_i$, and $\O[X]$ be a polynomial ring in a variable $X$ of degree $2$. Suppose that $m_i\cdot 1_\O$ is a unit in $\O$ for some $i\in I$, and
let $z_\al:=z(\bi)$ for $\bi\in \words_\al$ of the form $\bi=(i^{m_i},\dots)$, i.e. $\bi$ begins with $m_i$ lots of $i$'s. 
Then   
there exists a homogeneous isomorphism of graded algebras 
$$
R'_\al\otimes \O[X]\iso R_\al,\ a\otimes X\mapsto az_\al.
$$
\end{Proposition}
\begin{proof}
In view of the Basis Theorem, it suffices to prove that 
$$
\spa_\O(x_1, \dots,x_{d-1},z_\al)=\spa_\O(y_1,\dots,y_d),
$$
for which it is enough to see that $y_1\in \spa_\O(x_1, \dots,x_{d-1},z_\al)$. By (\ref{Ez}), 
$$y_1-(m_i\cdot 1_\O)^{-1}z_\al=\sum_{\bj\in \words_\al}\Big(y_1-(m_i\cdot 1_\O)^{-1}\sum_{1\leq r\leq d,\ j_r=i}y_r\Big)1_{\bj}.$$
Note that $|\{1\leq r\leq d,\ j_r=i\}|=m_i$ for all $\bj\in \words_\al$, so the sum of the coefficients of $y_k$'s in the expression above is zero, hence this expression is a linear combination of $x_1,\dots,x_{d-1}$. 
\end{proof}

\subsection{\boldmath Basic algebra $B_n$}\label{SB_n}

Let $B$ be the unital $\O$-algebra generated by the elements
$
\{e(+),e(-),b\}
$
subject only to the relations
$$
e(\pm)e(\mp)=0,\quad e(\pm)e(\pm)=e(\pm),\quad e(+)+e(-)=1,\quad e(\pm)b=be(\mp).
$$
In other words, $B$ is the path algebra of the quiver \ 
$
{\begin{picture}(24, 11)%
\put(0,2){\circle{4}}%
\put(20,2){\circle{4}}%
\put(10, -4){\makebox(0, 0)[b]{$\rightleftarrows$}}
\put(0, 6){\makebox(0, 0)[b]{$_{{+}}$}}%
\put(20, 6){\makebox(0, 0)[b]{$_{{-}}$}}%
\end{picture}}
$. 
Setting $\deg e(\pm):=0$, $\deg b:=1$ defines a grading on $B$. Note that $\{b^me(\si)\mid m\in \Z_{\geq 0},\ \si\in\{+,-\}\}$  is a basis of $B$. 

Let $P(\pm):=Be(\pm)$, and $L(\pm)$ be the rank one $\O$-module $\O\cdot v_{\pm}$ on the basis vector $v_\pm$, with the action of $B$ defined by 
$$
e(\tau)v_{\si}=\de_{\tau,\si}v_{\si},\quad bv_\si=0\qquad(\si,\tau\in \{+,-\}). 
$$ 
If $\O=F$, then $L(+)$ and $L(-)$ are the irreducible $B$-modules with projective covers $P(+)$ and $P(-)$ respectively.

More generally, let $B_n=B^{\otimes n}$. For $1\leq r\leq n$, set 
$$b_r:=1\otimes\dots\otimes1\otimes b\otimes 1\otimes \dots\otimes 1,
$$ 
with $b$ in the $r$th position, and for $\bsi=(\si_1,\dots,\si_n)\in\{\pm\}^n$ set
$$
e(\bsi):=e(\si_1)\otimes\dots\otimes e(\si_n).
$$
For $1\leq r\leq n$, denote
$$
\beps_r:=(+\dots+,-,+,\dots,+),
$$
with `$-$' in the $r$th position. For any $\bsi=(\si_1,\dots,\si_n)\in\{\pm\}^n$, we allow ourselves to multiply
$$
\beps_r\bsi:=(\si_1,\dots,\si_{r-1},-\si_r,\si_{r+1},\dots,\si_n).
$$
Then $B_n$ is generated by $\{e(\bsi),\ b_r\mid \bsi\in\{\pm\}^n,\ 1\leq r\leq n\}$ subject only to the relations
\begin{eqnarray}
e(\bsi)e(\btau)=\de_{\bsi,\btau}e(\bsi),\ 
\textstyle\sum_{\bsi\in\{\pm\}^n}e(\bsi)=1,
\label{ERB1}\\
 b_rb_s=b_sb_r,
 \label{ERB2}\\ 
 e(\bsi)b_r=b_re(\beps_r\bsi).\label{ERB3}
\end{eqnarray}
The algebra $B_n$ is graded with $\deg e(\bsi)=0$, $\deg b_r=1$, and has basis 
\begin{equation}\label{EBBasis}
\{b_1^{m_1}\dots b_n^{m_n}e(\bsi)\mid m_1,\dots,m_n\in \Z_{\geq 0},\ \bsi\in\{\pm\}^n\}.
\end{equation}

For $\bsi\in\{\pm\}^n$, let $P(\bsi):=B_ne(\bsi)$, and $L(\bsi)$ be the rank one $\O$-module $\O\cdot v_{\bsi}$ on the basis vector $b_\bsi$, with the action of $B_n$ defined by 
$$
e(\btau)v_{\bsi}=\de_{\btau,\bsi}v_{\bsi},\quad b_rv_\bsi=0\qquad(\btau\in \{\pm\}^n,\ 1\leq r\leq n). 
$$ 
If $\O=F$, then $\{L(\bsi)\mid \bsi\in\{\pm\}^n\}$ is a complete irredundant set of irreducible $B_n$-modules, and $P(\bsi)$ is a projective cover of $L(\bsi)$ for every $\bsi$.

For $n=1$, we have a linear minimal projective resolution $\BasicRes(\pm)$ of $L(\pm)$: 
\begin{equation*}
0\longrightarrow q P(\mp)\stackrel{\partial}{\longrightarrow} P(\pm)\longrightarrow L(\pm)\longrightarrow0,
\end{equation*}
where the map $\partial$ is the right  multiplication by $b$, i.e. 
$$
\partial(b^me(\mp))=b^{m+1}e(\pm)\qquad(m\in\Z_{\geq 0}).
$$

For a general $n$ and an arbitrary $\bsi\in\{\pm\}^n$, we have a linear minimal projective resolution $\BasicRes(\bsi):=\BasicRes(\si_1)\otimes\dots\otimes \BasicRes(\si_n)$ of $L(\bsi)$: 
\begin{equation}\label{BasicResn}
0\longrightarrow P_n(\bsi)
\stackrel{\partial_n}{\longrightarrow} P_{n-1}(\bsi)
\stackrel{\partial_{n-1}}{\longrightarrow}\dots\stackrel{\partial_1}{\longrightarrow}P_0(\bsi) \longrightarrow L(\bsi)\longrightarrow0,
\end{equation}
where 
$$
P_m(\bsi)=\bigoplus_{1\leq r_1<\dots<r_m\leq n}q^m  Be(\beps_{r_1}\dots\beps_{r_m}\bsi)\qquad(0\leq m\leq n),
$$
and the map $\partial_m$ is defined on the direct summand $Be(\beps_{r_1}\dots\beps_{r_m}\bsi)$ of $P_m(\bsi)$ as the right multiplication by $\sum_{k=1}^m(\prod_{l=1}^{k-1}\si_{r_l}){b_{r_k}}$ with $\prod_{l=1}^{k-1}\si_{r_l}{\in}\{\pm\}$ interpreted as $\pm1_\O$. In other words,
$$
\partial_m=\bigoplus_{1\leq r_1<\dots<r_m\leq n}\partial_m^{r_1,\dots,r_m}:P_m(\bsi)\to P_{m-1}(\bsi)
$$
where
$$
\partial_m^{r_1,\dots,r_m}=\sum_{k=1}^m\partial_{m,r_k}:Be(\beps_{r_1}\dots\beps_{r_m}\bsi)\to P_{m-1}(\bsi)
$$
for the homomorphism 
$$
\partial_{m,r_k}: Be(\beps_{r_1}\dots\beps_{r_m}\bsi)\to Be(\beps_{r_1}\dots\widehat{\beps_{r_k}}\dots\beps_{r_m}\bsi)
$$
which maps 
$$
b_1^{a_1}\dots b_n^{a_n}e(\beps_{r_1}\dots\beps_{r_m}\bsi)\in Be(\beps_{r_1}\dots\beps_{r_m}\bsi)
$$
to
$$(\prod_{l=1}^{k-1}\si_{r_l})b_1^{a_1}\dots b_n^{a_n}b_{r_k}e(\beps_{r_1}\dots\widehat{\beps_{r_k}}\dots\beps_{r_m}\bsi)\in Be(\beps_{r_1}\dots\widehat{\beps_{r_k}}\dots\beps_{r_m}\bsi).
$$

We  compute extensions between irreducible $B_n$-modules. First, for $n=1$ we have:

\begin{Lemma} \label{LB1EXT} 
We have 
\begin{align*}
\EXT^m_{B_1}(L(\pm),L(\pm))&=
\left\{
\begin{array}{ll}
\O &\hbox{if $m=0$,}\\
0 &\hbox{otherwise,}
\end{array}
\right.
\\ 
\EXT^m_{B_1}(L(\pm),L(\mp))&=
\left\{
\begin{array}{ll}
q^{-1}\O &\hbox{if $m=1$,}\\
0 &\hbox{otherwise.}
\end{array}
\right.
\end{align*}
\end{Lemma}
\begin{proof}
This is obtained by applying an appropriate functor $\HOM_{B_1}(-,L(\tau))$ to the resolution $\BasicRes(\si)$ for $\si,\tau\in\{\pm\}$ and computing cohomology. For example, let us compute $\EXT^m_{B_1}(L(\pm),L(\mp))$. An application of the functor $\HOM_{B_1}(-,L(\mp))$ to the resolution $\BasicRes(\pm)$ yields:
\begin{equation*}
0\longrightarrow \HOM_{B_1}(q P(\mp),L(\mp))\rangle\stackrel{\partial^*}{\longrightarrow} \HOM_{B_1}(P(\pm),L(\mp))\longrightarrow0.
\end{equation*}
But $\HOM_{B_1}(qP(\mp),L(\mp))\cong q^{-1} \O$ and $\HOM_{B_1}(P(\pm),L(\mp))= 0$, which immediately implies the required result. 
\end{proof}

Now we can deal with general $n$. 

\begin{Proposition} \label{PEXTB} 
Let $\bsi,\btau\in\{\pm\}^n$. Write $\bsi=\beps_{r_1}\dots \beps_{r_m}\btau$ for unique $1\leq r_1<\dots< r_m\leq n$. 
Then:
$$
\EXT^k_{B_n}(L(\bsi),L(\btau))=
\left\{
\begin{array}{ll}
q^{-k}\O &\hbox{if $k=m$;}
\\
0 &\hbox{\ otherwise.}
\end{array}
\right.
$$
\end{Proposition}
\begin{proof}
This obtained using K\"unneth formula and Lemma~\ref{LB1EXT}, since $B_r\simeq B_1\otimes\dots\otimes B_1$ and $L(\bsi)\simeq L(\si_1)\boxtimes\dots\boxtimes L(\si_n)$. 
\end{proof}

\subsection{Skew shapes}
Recall the standard notions concerning multipartitions from \S~\ref{SMotiv}. We are now going consider the special case $l=1$ and $\La=\La_0$, so  $l$-multipartitions are just partitions and boxes on the main diagonal have content $0$. 
If $\mu$ is a Young diagram contained in a Young diagram $\la$ then $\la\setminus\mu$ is called a {\em skew shape}. Skew shapes are identified up to the shifts along diagonals. For example, the following picture, with boxes marked with their contents, illustrates why $(7,7,4,1)\setminus(7,3,2,1)=(6,3)\setminus(2,1)$:
$$
\begin{tikzpicture}[scale=0.5,draw/.append style={thick,black}]
 \fill[blue!40](3.5,-1.5)rectangle(7.5,-0.5);
  \fill[blue!40](2.5,-2.5)rectangle(4.5,-1.5);
  \newcount\col
  \foreach\Row/\row in {{0,1,2,3,4,5,6}/0,{-1,0,1,2,3,4,5}/-1,{-2,-1,0,1}/-2,{-3}/-3} {
     \col=1
     \foreach\k in \Row {
        \draw(\the\col,\row)+(-.5,-.5)rectangle++(.5,.5);
        \draw(\the\col,\row)node{\k};
        \global\advance\col by 1
      }
   }
\end{tikzpicture}
\hspace{2mm}=\hspace{2 mm}
\begin{tikzpicture}[scale=0.5,draw/.append style={thick,black}]
 \fill[blue!40](2.5,-0.5)rectangle(6.5,0.5);
  \fill[blue!40](1.5,-1.5)rectangle(3.5,-0.5);
  \newcount\col
  \foreach\Row/\row in {{0,1,2,3,4,5}/0,{-1,0,1}/-1}
  {
     \col=1
     \foreach\k in \Row {
        \draw(\the\col,\row)+(-.5,-.5)rectangle++(.5,.5);
        \draw(\the\col,\row)node{\k};
        \global\advance\col by 1
      }
   }
                     --(3.5,-2.5)--(0.5,-2.5)--(0.5,-1.5)--(2.5,-1.5);
\end{tikzpicture}
$$
The given skew shape is determined by its Young diagram with box contents: 
$$
\young(:2345,{\tzero}1)
$$

To the positive root $\rho=\rho(k,l)$, we associate the set $\Shapes_\rho$ of all skew shapes containing exactly one box of each of the contents  $k,k+1,\dots,l$. For example, in the picture above the skew shape belongs to $\Shapes_{\rho(0,5)}$. 


As for partitions, for any $\la\in\Shapes_\rho$, a {\em $\la$-tableau} $\T$ is an allocation of the numbers $1,2,\dots,d$ into the boxes of $\la$ (recall that we have put $d:=\height(\rho)$). The symmetric group $\Si_d$ acts on the $\la$-tableaux by permutations of entries. For example $s_r\cdot \T$ is $\T$ with the entries $r$ and $r+1$ swapped. 
A $\la$-tableau is {\em standard} if  its entries increase along the columns from top to bottom and along the rows from left to right. The set of the standard $\la$-tableaux is denoted $\St(\la)$. 

The leading $\la$-tableaux $\T^\la$ is obtained by inserting the numbers $1,\dots,d$ into the boxes of $\la$ from left to right along the rows starting from the first row, then the second row, and so on. 
Given any $\T\in\St(\la)$, define $\bi^\T:=i_1 \dots i_d\in \words_\rho$, where $i_r$ is the content of the box occupied in $\T$ with $r$. 
Finally, set $\bi^\la:=\bi^{\T^\la}$.

There is a one-to-one correspondence between the set $\Shapes_\rho$ and the set $\{\pm\}^n$; in particular, $|\Shapes_\rho|=2^n$. To construct a bijection, number the boxes of a skew shape $\la\in\Shapes_\rho$ by the numbers $1,\dots,d$ from bottom-left to top-right, and for $r=1,\dots,n$, set $\si_r:=+$ if the $r$th box of $\la$ is not in the end of its row, and $\si_r:=-$ otherwise. This will produce a sequence $\bsi^\la=(\si_1,\dots,\si_n)\in\{\pm\}^n$. The following is elementary:

\begin{Lemma} \label{LBij}
The map $\Shapes_\rho\to\{\pm\}^n,\ \la\mapsto \bsi^\la$ is a bijection. 
\end{Lemma}

For any $1\leq r\leq n$ , define the {\em row splitting operator} $\linebr_r$ on the set $\Shapes_\rho$ as follows. Recall that we number the boxes with the numbers $1,\dots,d$ from bottom-left to top-right. 
Let $\la\in\Shapes_\rho$, and the $r$th box $A$ of $\la$ lie in the $m$th row; if $A$ is not in the end of the $m$th row, then $\linebr_r \la$ is the skew shape obtained from $\la$ by splitting its  $m$th row at $A$, so that $A$ is now in the end of its row in $\linebr_r \la$. On the other hand, if $A$ is at the end of the $m$th row, then $\linebr_r \la$ is the skew shape obtained from $\la$ by attaching the $(m-1)$st row to the end of the $m$th row. For example: 
$$
\linebr_4\cdot \young(:3456,12)\ =\  \young(::56,:34,12)
,\quad \linebr_2\cdot \young(:3456,12)\ =\ \young(123456)\ .
$$

The key properties of the row splitting  are as follows:

\begin{Lemma} \label{LPropRS}
Let $\la,\mu\in\Shapes_\rho$ and $1,\leq r,s\leq n$. Then:
\begin{enumerate}
\item[{\rm (i)}] $\linebr_r^2=\id$;
\item[{\rm (ii)}] $\linebr_r\linebr_s=\linebr_s\linebr_r$;
\item[{\rm (iii)}] $\bsi^{\linebr_r\la}=\beps_r\bsi^\la$. 
\item[{\rm (iv)}] There exist unique distinct numbers $r_1,\dots,r_l$ such that $1\leq r_1,\dots,r_l\leq n$ and $\la=\linebr_{r_1}\dots \linebr_{r_l}\mu$. 
\end{enumerate}
\end{Lemma}
\begin{proof}
Part (iii) is clear from the definitions. The rest follows from (iii) and Lemma~\ref{LBij}. 
\end{proof}

\subsection{\boldmath The elements $\psi_{\la,\mu}$}
Since $\rho\in\Phi_+$ has coefficients at most $1$ when decomposed as a linear combination of simple roots, the elements $\psi_{u}\in R_\rho$ are well-defined for all $u\in\Si_d$. Moreover, the action of the symmetric group $\Si_d$ on $\words_\rho$ is regular, and so for any $\la,\mu\in\Shapes_\rho$, there is a unique element $w(\la,\mu)\in\Si_d$ such that $w(\la,\mu)\cdot\bi^\mu=\bi^\la$. This yields well-defined elements $$\psi(\la,\mu):=\psi_{w(\la,\mu)}\qquad (\la,\mu\in\Shapes_\rho).$$ 
Note that $\psi(\la,\mu)1_{\bi^\mu}=1_{\bi^\la}\psi(\la,\mu)1_{\bi^\mu}=1_{\bi^\la}\psi(\la,\mu)$.

\begin{Lemma} \label{L230212}
Let $\mu\in \Shapes_\rho$, $r_1,\dots,r_l$ be distinct numbers such that $1\leq r_1,\dots,r_l\leq n$ and $\la=\linebr_{r_1}\dots \linebr_{r_l}\mu$. For $k=0,1\dots,l$, denote $\mu^{(k)}:=\prod_{m=1}^k\linebr_{r_m}\mu$, so that $\mu=\mu^{(0)}$ and $\la=\mu^{(l)}$. Then
$$
w(\la,\mu)=w(\mu^{(l)},\mu^{(l-1)})w(\mu^{(l-1)},\mu^{(l-2)})\dots w(\mu^{(1)},\mu^{(0)}).
$$
\end{Lemma}
\begin{proof}
The left hand side and the right hand side map $\bi^\mu$ to $\bi^\la$ and the symmetric group acts on $\words_\rho$ regularly. 
\end{proof}

\begin{Proposition} \label{PCrucial}
Let $\rho\in\Phi_+$, $d=\height(\rho)$, and $n=d-1$. 
\begin{enumerate}
\item[{\rm (i)}] Let $\nu\in\Shapes_\rho$ and $1\leq r\leq n$. Denote the row number of the $r$th box in $\nu$ by $m$. Let $a$ be the number of the leftmost box of the row $m$ in $\nu$ and $b$ be the number of the rightmost box of the row $m$ in $\nu$. 
\begin{enumerate}
\item[{\rm (a)}] If $r<b$, then $$\psi(\nu,\linebr_r\nu)\psi(\linebr_r\nu,\nu)1_{\bi^\nu}=x_{d-b+r-a+1}1_{\bi^\nu}.$$ 
\item[{\rm (b)}] If $r=b$, denote by $c$ the number of the rightmost box in the row $m-1$ of $\nu$. Then $$\psi(\nu,\linebr_r\nu)\psi(\linebr_r\nu,\nu)1_{\bi^\nu}=-\sum_{k=d-c+1}^{d-a}x_{k}1_{\bi^\nu}.$$ 
\end{enumerate}

\item[{\rm (ii)}] Let $\mu\in \Shapes_\rho$, $r_1,\dots,r_l$ be  distinct numbers such that $1\leq r_1,\dots,r_l\leq n$, and $\la=\linebr_{r_1}\dots \linebr_{r_l}\mu$. For $k=0,1\dots,l$, denote $\mu^{(k)}:=\prod_{m=1}^k\linebr_{r_m}\mu$, so that $\mu=\mu^{(0)}$ and $\la=\mu^{(l)}$. Then
\begin{equation}\label{EPCrucial}
\psi(\la,\mu)1_{\bi^\mu}=\psi(\la,\mu^{(l-1)})\psi(\mu^{(l-1)},\mu^{(l-2)})\dots \psi(\mu^{(1)},\mu)1_{\bi^\mu}.
\end{equation}
\end{enumerate}
\end{Proposition}
\begin{proof}
Assume without loss of generality that $\rho=\rho(1,d)$. Then the content of the $t$th box of any $\nu\in\Shapes_\rho$ is $t$ for all $t=1,\dots,d$. 

We first prove (i). 
The reader should keep in mind the following picture for $r,a,b$ in $\nu$:
$$
\begin{picture}(100, 80)%
\put(0,20){\young(a{\ },{\ })}%
\put(5,2){\vdots}%
\put(35, 38){\dots}
\put(53,32){\young({\ }r{\ })}
\put(115,32){\young(:{\ },{\ }b)}%
\put(97, 38){\dots}
\put(5,2){\vdots}%
\put(132, 65){\vdots}%
\put(-30, 32){$\nu=$}%
\end{picture}
$$ 
Also, if $r=b$, then:
$$
\begin{picture}(100, 80)%
\put(0,20){\young(a{\ }{\ },{\ })}%
\put(5,2){\vdots}%
\put(45, 38){\dots}
\put(65,32){\young(::{\ }{\ },{\ }{\ }r)}%
\put(5,2){\vdots}%
\put(45, 38){\dots}
\put(125, 51){\dots}
\put(145,44){\young({\ }c)}
\put(162, 65){\vdots}%
\put(-30, 32){$\nu=$}%
\end{picture}
$$

(a) If $r<b$, then, using the geometric presentation of the elements of $R_\rho$ introduced in \cite{KL1}, we have 
\begin{equation}\label{EMarrD1}
\psi(\linebr_r\nu ,\nu)1_{\bi^\nu}
= 
\begin{braid}\tikzset{baseline=3mm}
  \draw(0,4) node[above]{}--(0,0);
  \draw[dots] (1,4.4)--(2.2,4.4);
  \draw[dots] (1,0)--(2.2,0);
  \draw (3,4) node[above]{}--(3,0);
 \draw (4,4) node[above]{$a$}--(10,0);
 --(11,0);
  \draw[dots] (5.3,4.4)--(6.5,4.4);
  \draw[dots] (6,0)--(7.2,0);
  \draw[red] (8,4) node[above]{$r$}--(14,0);
 \draw[red](9.2,4) node[above]{$r\hspace{-.5mm}+\hspace{-.5mm}1$}--(4,0);
   \draw[dots] (11,4.4)--(12.2,4.4);
   \draw[dots] (11.5,0)--(12.7,0);
  \draw (14,4) node[above]{$b$}--(9,0);
  \draw (16.2,4)--(16.2,0); 
   \draw[dots] (17.2,4.4)--(18.4,4.4);
    \draw[dots] (17.2,0)--(18.4,0);
   \draw (19,4)--(19,0);
 \end{braid}\ .
\end{equation}
Then 
$$
\psi(\nu,\linebr_r\nu)\psi(\linebr_r\nu,\nu)1_{\bi^\nu}=
\begin{braid}\tikzset{baseline=3mm}
  \draw(0,5) node[above]{}--(0,0);
  \draw[dots] (1,5.4)--(2.2,5.4);
  \draw[dots] (1,0)--(2.2,0);
  \draw (3,5) node[above]{}--(3,0);
 \draw (4,5) node[above]{$a$}--(10,2.5)--(4,0);
  \draw[dots] (5.3,5.4)--(6.5,5.4);
  \draw[dots] (6,0)--(7.2,0);
  \draw[red] (8,5) node[above]{$r$}--(14,2.5)--(8,0);
 \draw[red](9.2,5) node[above]{$r\hspace{-.5mm}+\hspace{-.5mm}1$}--(4,2.5)--(9.2,0);
   \draw[dots] (11,5.4)--(12.2,5.4);
   \draw[dots] (11.5,0)--(12.7,0);
  \draw (14,5) node[above]{$b$}--(9,2.5)--(14,0);
  \draw (16.2,5)--(16.2,0); 
   \draw[dots] (17.2,5.4)--(18.4,5.4);
    \draw[dots] (17.2,0)--(18.4,0);
   \draw (19,5)--(19,0);
 \end{braid}\ .
$$ 
Using defining relations in $R_\rho$, we see that this element
equals  
$$
\begin{braid}\tikzset{baseline=3mm}
  \draw(0,4) node[above]{}--(0,0);
  \draw[dots] (1,4.4)--(2.2,4.4);
  \draw[dots] (1,0)--(2.2,0);
  \draw (3,4) node[above]{}--(3,0);
 \draw (4,4) node[above]{$a$}--(4,0);
  \draw[dots] (4.8,4.4)--(6,4.4);
  \draw[dots] (4.8,0)--(6,0);
  \draw(6.9,4) node[above]{$r\hspace{-.5mm}-\hspace{-.5mm}1$}--(6.9,0);
  \draw[red] (8,4) node[above]{$r$}--(9.8,2)--(8,0);
 \draw[red](9.2,4) node[above]{$r\hspace{-.5mm}+\hspace{-.5mm}1$}--(7.5,2)--(9.2,0);
 \draw(10.5,4) node[above]{$r\hspace{-.5mm}+\hspace{-.5mm}2$}--(10.5,0);
   \draw[dots] (11.6,4.4)--(12.8,4.4);
   \draw[dots] (11.6,0)--(12.8,0);
  \draw (14,4) node[above]{$b$}--(14,0);
  \draw (16.2,4)--(16.2,0); 
   \draw[dots] (17.2,4.4)--(18.4,4.4);
    \draw[dots] (17.2,0)--(18.4,0);
   \draw (19,4)--(19,0);
 \end{braid}\ ,
$$ 
which equals
\begin{align*}
\begin{braid}\tikzset{baseline=3mm}
  \draw(0,3) node[above]{}--(0,0);
  \draw[dots] (1,3.4)--(2.2,3.4);
  \draw[dots] (1,0)--(2.2,0);
  \draw(3.7,3) node[above]{$r\hspace{-.5mm}-\hspace{-.5mm}1$}--(3.7,0);
  \draw (5,3) node[above]{$r$}--(5,0);
 \draw(6.2,3) node[above]{$r\hspace{-.5mm}+\hspace{-.5mm}1$}--(6.2,0);
 \draw(7.5,3) node[above]{$r\hspace{-.5mm}+\hspace{-.5mm}2$}--(7.5,0);
   \draw[dots] (8.6,3.4)--(9.8,3.4);
   \draw[dots] (8.6,0)--(9.8,0);
   \draw (10,3)--(10,0); 
   \greendot(5,1.5);
 \end{braid}
 \quad-\quad
\begin{braid}\tikzset{baseline=3mm}
  \draw(0,3) node[above]{}--(0,0);
  \draw[dots] (1,3.4)--(2.2,3.4);
  \draw[dots] (1,0)--(2.2,0);
  \draw(3.7,3) node[above]{$r\hspace{-.5mm}-\hspace{-.5mm}1$}--(3.7,0);
  \draw (5,3) node[above]{$r$}--(5,0);
 \draw(6.2,3) node[above]{$r\hspace{-.5mm}+\hspace{-.5mm}1$}--(6.2,0);
 \draw(7.5,3) node[above]{$r\hspace{-.5mm}+\hspace{-.5mm}2$}--(7.5,0);
   \draw[dots] (8.6,3.4)--(9.8,3.4);
   \draw[dots] (8.6,0)--(9.8,0);
   \draw (10,3)--(10,0);
 \greendot(6.2,1.5);
 \end{braid}\ ,
\end{align*}
which is  $(y_{d-b+r-a+1}-y_{d-b+r-a+2})1_{\bi^\nu}=x_{d-b+r-a+1}1_{\bi^\nu}$.

(b) If $r=b$, then  
\begin{equation}\label{EMarrD2}
\psi(\linebr_r\nu ,\nu)1_{\bi^\nu}
= 
\begin{braid}\tikzset{baseline=3mm}
  \draw(0,4) node[above]{}--(0,0);
  \draw[dots] (1,4.4)--(2.2,4.4);
  \draw[dots] (1,0)--(2.2,0);
  \draw (3,4) node[above]{}--(3,0);
 \draw[red] (4,4) node[above]{$r\hspace{-.5mm}+\hspace{-.5mm}1$}--(10,0);
 --(11,0);
  \draw[dots] (5.3,4.4)--(6.5,4.4);
  \draw[dots] (6,0)--(7.2,0);
  \draw (8,4) node[above]{$c$}--(14,0);
 \draw(9.2,4) node[above]{$a$}--(4,0);
   \draw[dots] (11,4.4)--(12.2,4.4);
   \draw[dots] (11.5,0)--(12.7,0);
  \draw[red] (14,4) node[above]{$r$}--(9,0);
  \draw (16.2,4)--(16.2,0); 
   \draw[dots] (17.2,4.4)--(18.4,4.4);
    \draw[dots] (17.2,0)--(18.4,0);
   \draw (19,4)--(19,0);
 \end{braid}\ .
 \end{equation}
Then
$$\psi(\nu ,\linebr_r\nu)\psi(\linebr_r\nu ,\nu)1_{\bi^\nu}
= 
\begin{braid}\tikzset{baseline=3mm}
  \draw(0,5) node[above]{}--(0,0);
  \draw[dots] (1,5.4)--(2.2,5.4);
  \draw[dots] (1,0)--(2.2,0);
  \draw (3,5) node[above]{}--(3,0);
 \draw[red] (3.8,5) node[above]{$r\hspace{-.5mm}+\hspace{-.5mm}1$}--(8.8,2.5)--(3.8,0);
  \draw (5,5) node[above]{$r\hspace{-.5mm}+\hspace{-.5mm}2$}--(10,2.5)--(5,0);
  \draw[dots] (6,5.4)--(7.2,5.4);
  \draw[dots] (6.2,0)--(7.4,0);
  \draw (7.5,5) node[above]{$c$}--(12.5,2.5)--(7.5,0);
 \draw(8.7,5) node[above]{$a$}--(4,2.5)--(8.7,0);
   \draw[dots] (9.5,5.4)--(10.7,5.4);
   \draw[dots] (9.5,0)--(10.7,0);
  \draw(11.5,5) node[above]{$r\hspace{-.5mm}-\hspace{-.5mm}1$}--(6.5,2.5)--(11.5,0);
  \draw[red] (12.5,5) node[above]{$r$}--(7.5,2.5)--(12.5,0);
  \draw (13.2,5)--(13.2,0); 
   \draw[dots] (14,5.4)--(15.2,5.4);
    \draw[dots] (14,0)--(15.2,0);
   \draw (15.8,5)--(15.8,0);
 \end{braid}\ ,
$$
which equals
\begin{align*}
&\begin{braid}\tikzset{baseline=3mm}
  \draw(0,4) node[above]{}--(0,0);
  \draw[dots] (.5,4.4)--(1.7,4.4);
  \draw[dots] (.5,0)--(1.7,0);
  \draw (2,4) node[above]{}--(2,0);
 \draw[red] (2.8,4) node[above]{$r\hspace{-.5mm}+\hspace{-.5mm}1$}--(7,2)--(2.8,0);
  \draw (4,4) node[above]{$r\hspace{-.5mm}+\hspace{-.5mm}2$}--(9,2)--(4,0);
  \draw[dots] (5,4.4)--(6.2,4.4);
  \draw[dots] (5.2,0)--(6.4,0);
  \draw (6.5,4) node[above]{$c$}--(11.1,2)--(6.5,0);
 \draw(7.4,4) node[above]{$a$}--(3,2)--(7.4,0);
   \draw[dots] (8.1,4.4)--(9.3,4.4);
   \draw[dots] (8.1,0)--(9.3,0);
  \draw(10.1,4) node[above]{$r\hspace{-.5mm}-\hspace{-.5mm}1$}--(5.5,2)--(10.1,0);
  \draw[red] (11.1,4) node[above]{$r$}--(7.5,2)--(11.1,0);
  \draw (11.8,4)--(11.8,0); 
   \draw[dots] (12.2,4.4)--(13.4,4.4);
    \draw[dots] (12.2,0)--(13.4,0);
   \draw (13.8,4)--(13.8,0);
   \greendot(7.6,2);
 \end{braid}\ -\ 
 \begin{braid}\tikzset{baseline=3mm}
  \draw(0,4) node[above]{}--(0,0);
  \draw[dots] (.5,4.4)--(1.7,4.4);
  \draw[dots] (.5,0)--(1.7,0);
  \draw (2,4) node[above]{}--(2,0);
 \draw[red] (2.8,4) node[above]{$r\hspace{-.5mm}+\hspace{-.5mm}1$}--(7,2)--(2.8,0);
  \draw (4,4) node[above]{$r\hspace{-.5mm}+\hspace{-.5mm}2$}--(9,2)--(4,0);
  \draw[dots] (5,4.4)--(6.2,4.4);
  \draw[dots] (5.2,0)--(6.4,0);
  \draw (6.5,4) node[above]{$c$}--(11.1,2)--(6.5,0);
 \draw(7.4,4) node[above]{$a$}--(3,2)--(7.4,0);
   \draw[dots] (8.1,4.4)--(9.3,4.4);
   \draw[dots] (8.1,0)--(9.3,0);
  \draw(10.1,4) node[above]{$r\hspace{-.5mm}-\hspace{-.5mm}1$}--(5.5,2)--(10.1,0);
  \draw[red] (11.1,4) node[above]{$r$}--(7.5,2)--(11.1,0);
  \draw (11.8,4)--(11.8,0); 
   \draw[dots] (12.2,4.4)--(13.4,4.4);
    \draw[dots] (12.2,0)--(13.4,0);
   \draw (13.8,4)--(13.8,0);
 \greendot(6.8,2);
 \end{braid}\ 
 \\
 =&\begin{braid}\tikzset{baseline=3mm}
  \draw(0,4) node[above]{}--(0,0);
  \draw[dots] (.5,4.4)--(1.7,4.4);
  \draw[dots] (.5,0)--(1.7,0);
  \draw (2,4) node[above]{}--(2,0);
 \draw[red] (2.8,4) node[above]{$r\hspace{-.5mm}+\hspace{-.5mm}1$}--(7,2)--(2.8,0);
  \draw (4,4) node[above]{$r\hspace{-.5mm}+\hspace{-.5mm}2$}--(9,2)--(4,0);
  \draw[dots] (5,4.4)--(6.2,4.4);
  \draw[dots] (5.2,0)--(6.4,0);
  \draw (6.5,4) node[above]{$c$}--(11.1,2)--(6.5,0);
 \draw(7.4,4) node[above]{$a$}--(3,2)--(7.4,0);
   \draw[dots] (8.1,4.4)--(9.3,4.4);
   \draw[dots] (8.1,0)--(9.3,0);
  \draw(10.1,4) node[above]{$r\hspace{-.5mm}-\hspace{-.5mm}1$}--(5.5,2)--(10.1,0);
  \draw[red] (11.1,4) node[above]{$r$}--(7.5,2)--(11.1,0);
  \draw (11.8,4)--(11.8,0); 
   \draw[dots] (12.2,4.4)--(13.4,4.4);
    \draw[dots] (12.2,0)--(13.4,0);
   \draw (13.8,4)--(13.8,0);
   \greendot(11,4);
 \end{braid}\ -\ 
 \begin{braid}\tikzset{baseline=3mm}
  \draw(0,4) node[above]{}--(0,0);
  \draw[dots] (.5,4.4)--(1.7,4.4);
  \draw[dots] (.5,0)--(1.7,0);
  \draw (2,4) node[above]{}--(2,0);
 \draw[red] (2.8,4) node[above]{$r\hspace{-.5mm}+\hspace{-.5mm}1$}--(7,2)--(2.8,0);
  \draw (4,4) node[above]{$r\hspace{-.5mm}+\hspace{-.5mm}2$}--(9,2)--(4,0);
  \draw[dots] (5,4.4)--(6.2,4.4);
  \draw[dots] (5.2,0)--(6.4,0);
  \draw (6.5,4) node[above]{$c$}--(11.1,2)--(6.5,0);
 \draw(7.4,4) node[above]{$a$}--(3,2)--(7.4,0);
   \draw[dots] (8.1,4.4)--(9.3,4.4);
   \draw[dots] (8.1,0)--(9.3,0);
  \draw(10.1,4) node[above]{$r\hspace{-.5mm}-\hspace{-.5mm}1$}--(5.5,2)--(10.1,0);
  \draw[red] (11.1,4) node[above]{$r$}--(7.5,2)--(11.1,0);
  \draw (11.8,4)--(11.8,0); 
   \draw[dots] (12.2,4.4)--(13.4,4.4);
    \draw[dots] (12.2,0)--(13.4,0);
   \draw (13.8,4)--(13.8,0);
 \greendot(2.8,4);
 \end{braid}
 \\
 =&\begin{braid}\tikzset{baseline=3mm}
  \draw(0,4) node[above]{}--(0,0);
  \draw[dots] (.5,4.4)--(1.7,4.4);
  \draw[dots] (.5,0)--(1.7,0);
  \draw (2,4) node[above]{}--(2,0);
 \draw(2.8,4) node[above]{$r\hspace{-.5mm}+\hspace{-.5mm}1$}--(2.8,0);
  \draw (4,4) node[above]{$r\hspace{-.5mm}+\hspace{-.5mm}2$}--(4,0);
  \draw[dots] (5,4.4)--(6.2,4.4);
  \draw[dots] (5.2,0)--(6.4,0);
  \draw (6.5,4) node[above]{$c$}--(6.5,0);
 \draw(7.4,4) node[above]{$a$}--(7.4,0);
   \draw[dots] (8.1,4.4)--(9.3,4.4);
   \draw[dots] (8.1,0)--(9.3,0);
  \draw(10.1,4) node[above]{$r\hspace{-.5mm}-\hspace{-.5mm}1$}--(10.1,0);
  \draw (11.1,4) node[above]{$r$}--(11.1,0);
  \draw (11.8,4)--(11.8,0); 
   \draw[dots] (12.2,4.4)--(13.4,4.4);
    \draw[dots] (12.2,0)--(13.4,0);
   \draw (13.8,4)--(13.8,0);
   \greendot(11.1,4);
 \end{braid}\ 
 -\ 
 \begin{braid}\tikzset{baseline=3mm}
  \draw(0,4) node[above]{}--(0,0);
  \draw[dots] (.5,4.4)--(1.7,4.4);
  \draw[dots] (.5,0)--(1.7,0);
  \draw (2,4) node[above]{}--(2,0);
 \draw (2.8,4) node[above]{$r\hspace{-.5mm}+\hspace{-.5mm}1$}--(2.8,0);
  \draw (4,4) node[above]{$r\hspace{-.5mm}+\hspace{-.5mm}2$}--(4,0);
  \draw[dots] (5,4.4)--(6.2,4.4);
  \draw[dots] (5.2,0)--(6.4,0);
  \draw (6.5,4) node[above]{$c$}--(6.5,0);
 \draw(7.4,4) node[above]{$a$}--(7.4,0);
   \draw[dots] (8.1,4.4)--(9.3,4.4);
   \draw[dots] (8.1,0)--(9.3,0);
  \draw(10.1,4) node[above]{$r\hspace{-.5mm}-\hspace{-.5mm}1$}--(10.1,0);
  \draw (11.1,4) node[above]{$r$}--(11.1,0);
  \draw (11.8,4)--(11.8,0); 
   \draw[dots] (12.2,4.4)--(13.4,4.4);
    \draw[dots] (12.2,0)--(13.4,0);
   \draw (13.8,4)--(13.8,0);
 \greendot(2.8,4);
 \end{braid} \ ,
 \end{align*}
 which is $(-y_{d-c+1}+y_{d-a+1})1_{\bi^\nu}=-\sum_{k=d-c+1}^{d-a}x_{k}1_{\bi^\nu}$.

(ii) Note from (\ref{EMarrD1}) and (\ref{EMarrD2}) that $r$ and $r+1$ is the only pair of neighboring entries in $\bi^\nu$ that get permuted in the pictures above (the corresponding strings are colored red). This can be restated as the claim that we can write 
$\psi(\linebr_r\nu ,\nu)1_{\bi^\nu}$ as a product of elements of the form $\psi_t1_{\bj}$ with $|j_t-j_{t+1}|>1$ for all  factors but one, and for that exceptional factor we have $\{j_t,t_{t+1}\}=\{r,r+1\}$. 

Applying this observation repeatedly to the product in the right hand side of (\ref{EPCrucial}), we conclude that it can be written as a product of elements of the form $\psi_t1_{\bj}$ with $|j_t-j_{t+1}|>1$ for all but $l$ special factors, which are of the form, $\psi_{t_k}1_{\bj^{(k)}}$ for $k=1,\dots,l$, and we have $\{j^{(k)}_{t_k},j^{(k)}_{t_k+1}\}=\{r_k,r_k+1\}$ for all $k=1,\dots,l$. In particular, this means that we can get to the reduced decomposition of the right hand side using only braid  relations of the form $\psi_t\psi_{t+1}\psi_t1_{\bj}=\psi_{t+1}\psi_t\psi_{t+1}1_{\bj}$, and the quadratic relations of the form 
$\psi_t^21_{\bj}=1_{\bj}$. Now part (ii) follows from Lemma~\ref{L230212}. 
\end{proof}

\subsection{Irreducibles and PIMs for cuspidal blocks}
As a special case of the main result of \cite{KRhomog}, we can describe all irreducible $R_\rho$-modules explicitly in spirit of Theorem~\ref{TYoung}. Given a skew shape $\la\in\Shapes_\rho$, consider the free $\O$-module  
$$
L^\la=\bigoplus_{\T\in\St(\la)}\O\cdot v_\T
$$
concentrated in degree~$0$ with $\O$-basis labelled by the standard $\la$-tableaux. The action of $R_\rho$ on $L^\la$ is defined by
\begin{equation}\label{Ehomog}
1_{\bj}v_\T=\de_{\bj,\bi^\T},\quad y_rv_\T=0,\quad \psi_t v_\T=
\left\{
\begin{array}{ll}
v_{s_r\cdot\T} &\hbox{if $s_r\cdot\T\in\St(\la)$,}\\
0 &\hbox{otherwise}
\end{array}
\right.
\end{equation}
for all admissible $\bj,r,t$. 

Moreover, suppose that $\la$ has $m$ rows. For $k=1,\dots,m$, define $\be_k$ to be the sum of the simple roots $\al_i$ where $i$ runs over the contents of the boxes in the row $k$ of $\la$. Then 
$$
\pi(\la)=(\be_1,\dots,\be_m)\in\Pi(\rho).
$$

\begin{Theorem} \label{Thomog} 
Let $\O=F$. Then: 
\begin{enumerate}
\item[{\rm (i)}] The formulas (\ref{Ehomog}) define a structure of an irreducible $R_\rho$-module on $L^\la$, which is also irreducible on restriction to $R'_\rho$. 
\item[{\rm (ii)}] $\{L^\la\mid \la\in \Shapes_\rho\}$ is a complete and irredundant set of irreducible $R_\rho$-modules up to degree zero isomorphism and degree shift.
\item[{\rm(iii)}] For any $\la\in\Shapes_\rho$, we have $L^\la\simeq L(\pi(\la))$. 
\item[{\rm (iv)}] For any $\la\in\Shapes_\rho$, we have $\CH L^\la=\sum_{\T\in\St(\la)}\bi^{\T}$. 
\item[{\rm (v)}] If $\T\in\St(\la)$ and $\Stab\in\St(\mu)$ for some $\la,\mu\in\Shapes_\rho$ then $\bi^\T=\bi^\Stab$ if and only if $\la=\mu$ and $\T=\Stab$. 
\end{enumerate}
\end{Theorem}
\begin{proof}
In view of \cite[Theorem 3.4]{KRhomog}, all irreducible $R_\rho$-modules are homogeneous. Now parts (i) and (ii) follow from \cite[Theorem 3.6]{KRhomog}. 
Part (iii) follows from part (ii) and Theorem~\ref{THeadIrr}. Finally, (iv) is clear from definition, while (v) follows again from \cite[Theorem 3.4]{KRhomog}. 
\end{proof}


\begin{Lemma} \label{LPIMhomog}
Let $\la\in\Shapes_\rho$. Then $E(\bi^\la)$ is a projective cover of $L^\la$ as $R_\rho$-modules or $R'_\rho$-modules. 
\end{Lemma}
\begin{proof}
Since $\bi^\la$ is a weight of $L^\la$, and the corresponding weight space generates $L^\la$, the projective module $E(\bi^\la)$ surjects onto $L^\la$. It remains to prove that $1_{\bi^\la}$ is a primitive idempotent, for which, using change of scalars, we may assume that $\O=F$. In that case, it suffices to notice that the head of $E(\bi^\la)$ is isomorphic to $L(\la)$. This fact follows from $\HOM_{R_\rho}(E(\bi^\la),L^\la)\simeq 1_{\bi^\la}L^\la\simeq F$ and $\HOM_{R_\rho}(E(\bi^\la),L^\mu)\simeq 1_{\bi^\la}L^\mu=0$ for $\mu\neq \la$, using Theorem~\ref{Thomog}(v). 
The proof for $R'_\rho$ is the same. 
\end{proof}

\begin{Corollary} \label{CComplId} 
$\{1_{\bi^\la}\mid\la\in\Shapes_\rho\}$ is a complete set of inequivalent primitive idempotents in $R_\rho$ and $R'_\rho$
\end{Corollary}

\subsection{Basic algebras of cuspidal blocks}
By Corollary~\ref{CComplId}, $\{1_{\bi^\la}\mid\la\in\Shapes_\rho\}$ is a complete set of inequivalent primitive idempotents in $R_\rho$ and $R'_\rho$. So, setting,
\begin{equation}\label{Ee}
e:=\sum_{\la\in\Shapes_\rho}1_{\bi^\la},
\end{equation}
we see that
$eR_\rho e$ is a basic algebra Morita equivalent to $R_\rho$ and $eR'_\rho e$ is a basic algebra Morita equivalent to $R'_\rho$.

\begin{Lemma} \label{LBasicBasis}
 $\{x_1^{m_1}\dots x_n^{m_n}\psi(\la,\mu)1_{\bi^\mu}\mid \la,\mu\in\Shapes_\rho,\ m_1,\dots,m_n\in\Z_{\geq 0}\}$ is an $\O$-basis of $eR'_\rho e$. 
\end{Lemma}
\begin{proof}
This follows from Theorem~\ref{TBasis} and the defining relations for $R_\rho$. 
\end{proof}

Recall the bijection $\Shapes_\rho\iso\{\pm\}^n,\ \la\mapsto \bsi^\la$ and the algebra $B_n$ from \S\ref{SB_n}.

\begin{Theorem} \label{EIsoB}
There is a homogeneous isomorphism of graded algebras
$$
\iota:B_n\to eR'_\rho e,\ e(\bsi^\la)\mapsto 1_{\bi^\la}, \ b_re(\bsi^\la) \mapsto \psi(\linebr_r\la,\la)1_{\bi^\la}
$$
for all $\la\in\Shapes_\rho$ and $1\leq r\leq n$. 
\end{Theorem}
\begin{proof}
To prove that there is a homomorphism $\iota$ as in the statement of the theorem, we check the defining relations of $B_n$ for the images of its generators. The relations (\ref{ERB1}) are clear since $1_{\bi^\la}$'s are orthogonal idempotents which sum to the identity $e$ in the algebra $eR'_\rho e$. 

Note that by linearity we must have $\iota(b_r)= \sum_{\mu\in\Shapes_\rho}\psi(\linebr_r\mu,\mu)1_{\bi^\mu}$. Moreover, by Lemma~\ref{LPropRS}(iii), $\iota\big(e(\beps_r\bsi^\la)\big)=1_{\bi^{\linebr_r\la}}$. So to check the relation (\ref{ERB3}), we have to prove that for all $\la\in\Shapes_\rho$ and $1\leq r\leq n$ we have  that 
$$1_{\bi^\la}\sum_{\mu\in\Shapes_\rho}\psi(\linebr_r\mu,\mu)1_{\bi^\mu}=\Big(\sum_{\mu\in\Shapes_\rho}\psi(\linebr_r\mu,\mu)1_{\bi^\mu}\Big)1_{\bi^{\linebr_r\la}}.$$
But in view of the relation (\ref{R2PsiE}), both sides are equal to $\psi(\la,\linebr_r\la)1_{\bi^{\linebr_r\la}}$.

Now, to check the relation (\ref{ERB2}), it suffices to verify that
$\iota(b_rb_se(\bsi^\la))=\iota(b_sb_re(\bsi^\la))$ for all admissible $r,s,\la$, or
$$
\psi(\linebr_s \linebr_r\la,\linebr_r \la)\psi(\linebr_r\la ,\la)1_{\bi^\la}=
\psi(\linebr_r \linebr_s\la,\linebr_s \la)\psi_{\linebr_s\la ,\la}1_{\bi^\la}.
$$
Of course, we may assume that $r\neq s$. Then, by Proposition~\ref{PCrucial}(ii), the left hand side is equal to $\psi(\linebr_s \linebr_r\la,\la)1_{\bi^\la}$ and the right hand side is equal to $\psi(\linebr_r \linebr_s\la,\la)1_{\bi^\la}$. It remains to apply Lemma~\ref{LPropRS}(ii). 

Finally, to prove that $\iota$ is an isomorphism, we show that the basis (\ref{EBBasis}) of $B_n$ is mapped by $\iota$ to a basis of $eR'_\rho e$.  
Note that for $1\leq r\leq n$, we have 
$$
\iota(b_r^2e(\bsi^\la))=\psi(\la,\linebr_r\la)\psi(\linebr_r\la,\la)1_{\bi^\la}. 
$$
The right hand side has been computed in Proposition~\ref{PCrucial}(i) as follows. Denote the row number of the $r$th box in $\la$ by $m$. Let $a(r)$ be the number of the leftmost box of the row $m$ in $\la$ and $b(r)$ be the number of the rightmost box of the row $m$ in $\la$. If $r<b(r)$, then by Proposition~\ref{PCrucial}(i)(a), we have 
$$\iota(b_r^2e(\bsi^\la))=x_{d-b(r)+r-a(r)+1}1_{\bi^\la}.$$ 
If $r=b(r)$, denote by $c(r)$ the number of the rightmost box in the row $m-1$ of $\la$. Then by Proposition~\ref{PCrucial}(i)(b), we have 
$$\iota(b_r^2e(\bsi^\la))=-\sum_{k=d-c(r)+1}^{d-a(r)}x_{k}1_{\bi^\la}.$$ 
It is now easy to see that 
$$\spa_\O\big(\iota(b_1^2e(\bsi^\la)),\dots,\iota(b_n^2e(\bsi^\la))\big)=\spa_\O(x_11_{\bi^\la},\dots,x_n1_{\bi^\la}).
$$
Therefore 
$$
\{\iota(b_1^{2m_1}\dots b_n^{2m_n}e(\bsi^\la))\mid m_1,\dots,m_n\in \Z_{\geq 0}\}
$$
is an $\O$-basis of 
$$
\spa_\O(x_1^{m_1}\dots x_n^{m_n}1_{\bi^\la}\mid m_1,\dots,m_n\in \Z_{\geq 0}).
$$
Moreover, let $\mu\in \Shapes_\rho$. Let $1\leq r_1,\dots,r_l\leq n$ be distinct numbers such that $\la=\linebr_{r_1}\dots \linebr_{r_l}\mu$, see Lemma~\ref{LPropRS}(iv). It follows from Proposition~\ref{PCrucial}(ii) that 
$$
\psi(\la,\mu)1_{\bi^\mu}=\iota(b_{r_1}\dotsb_{r_l}e(\bsi^\mu)).
$$
It remains to apply Lemma~\ref{LBasicBasis}. 
\end{proof}


\begin{Theorem} \label{TMorita} 
For $\rho\in\Phi_+$ with $\height(\rho)=n+1$, there is a Morita equivalence $\funF:\Mod{B_n}\to\Mod{R'_\rho}$ such that 
$$\funF(P(\bsi^\la))= E(\bi^\la),\quad \funF(L(\bsi^\la))=L^\la$$ 
for all $\la\in\Shapes_\rho$. 
\end{Theorem}
\begin{proof}
Let $e$ be as in (\ref{Ee}). Since $\{1_{\bi^\la}\mid\la\in\Shapes_\rho\}$ is a complete system of orthogonal primitive idempotents by Corollary~\ref{CComplId}, we have Morita equivalence 
$$
\funF:\Mod{eR'_\rho e}\to\Mod{R'_\rho}, \ V\mapsto R'_\rho e\otimes_{eR'_\rho e}V.
$$
By Theorem~\ref{EIsoB}, there is an isomorphism $\iota: B_n\to e R'_\rho e$, which maps $e(\bsi^\la)$ to $1_{\bi^\la}$ for all $\la\in\Shapes_\rho$. The result follows. 
\end{proof}

The next corollary should be compared to Theorem~\ref{source}. 

\begin{Corollary} 
Let $\la,\mu\in\Shapes_\rho$. Write $\la=\linebr_{r_1}\dots \linebr_{r_m}\mu$ for unique $1\leq r_1<\dots< r_m\leq n$. 
Then:
\begin{align*}
\EXT^k_{R_\rho'}(L^\la,L^\mu)&\cong
\left\{
\begin{array}{ll}
q^{-k}\O &\hbox{if $k=m$;}
\\
0 &\hbox{\ otherwise.}
\end{array}
\right.
\\
\EXT^k_{R_\rho}(L^\la,L^\mu)&\cong
\left\{
\begin{array}{ll}
q^{-k}\O &\hbox{if $k=m$;}
\\
q^{-k+3}\O &\hbox{if $k=m+1$;}
\\
0 &\hbox{\ otherwise.}
\end{array}
\right.
\end{align*}
\end{Corollary}
\begin{proof}
The first equality follows from Theorem~\ref{TMorita} and Proposition~\ref{PEXTB}. The second equality follows from the first and the K\"unneth Formula, since $R_\rho\cong R_\rho'\otimes \O[X]$ with $\deg X=2$, thanks to Proposition~\ref{Pz}.  
\end{proof}

\subsection{Resolutions for cuspidal blocks} Let again $\rho\in\Phi_+$ be a positive root of height $d=n+1$.  Even though we have already computed extensions between the irreducible modules for the corresponding cuspidal block, we still need projective resolutions of them. 

Recall that we have defined line-breaking operators $\linebr_r$ for all $1\leq r<d$. Now define also $\linebr_d$ to be the trivial operator: $\linebr_d\la:=\la$ for all $\la\in\Shapes_\rho$. For a subset $T=\{t_1<\dots<t_m\}\subseteq [1,d]$ and $1\leq r\leq d$ define
\begin{align}
\linebr_T&:=\linebr_{t_1}\dots\linebr_{t_m}, 
\label{EfT}
\\
s(T)&:=
\left\{
\begin{array}{ll}
|T|+1 &\hbox{if $d\in T$}\\
|T| &\hbox{if $d\not\in T$}
\end{array}
\right.,
\\ 
\si(T,r)&:=(-1)^{|\{t\in T\mid t<r\}|}.
\end{align}

Now fix $\la\in\Shapes_\rho$. 
The resolution $\ResolPrime(\la)$ is the projective resolution of the $R'_\rho$-module $L^\la$ obtained by applying the Morita equivalence $\funF$ of Theorem~\ref{TMorita} to the resolution $\BasicRes(\bsi^\la)$, i.e. $\ResolPrime(\la)$ is: 
\begin{equation}\label{ERes'la}
0\longrightarrow P'_n\stackrel{\partial'_n}{\longrightarrow} P'_{n-1}\stackrel{\partial'_{n-1}}{\longrightarrow}\dots\stackrel{\partial'_1}{\longrightarrow}P'_0 \longrightarrow L^\la\longrightarrow0, 
\end{equation}
where 
$$
P'_m=\bigoplus_{T\subseteq [1,n],\ |T|=m} q^m E'(\bi^{\linebr_{T}\la})\qquad(0\leq m\leq n),
$$
and the map $\partial'_m$ is defined on the direct summand $q^m E'(\bi^{\linebr_{T}\la})$ of $P'_m$ as the right multiplication by 
$\sum_{t\in T}\si(T,t)\psi(\linebr_{T}\la,\linebr_{T\setminus\{t\}}\la).$ 

Finally, by Proposition~\ref{Pz}, we have $R_\rho\simeq R'_\rho\otimes \O[X]$ where $X$ is an indeterminate of degree $2$, which under the isomorphism corresponds to the central element $z_\rho\in R_\rho$. Now, the irreducible $R_\rho$-module $L^\la$ can be considered as the outer tensor product $L^\la\boxtimes \O$ where $L^\la$ is the restriction from $R_\rho$ to $R'_\rho$ of $L^\la$, and $\O$ is the $\O[X]$ module of rank $1$ with the trivial action of $X$. Tensoring $\ResolPrime(\la)$ with the following resolution $\Kos$ of $\O$:
$$
0\longrightarrow \O[X]\stackrel{X}{\longrightarrow}\O[X]\longrightarrow \O\longrightarrow 0,
$$
we obtain the resolution $\Resol(\la):=\ResolPrime(\la)\otimes \Kos$:
\begin{equation}\label{EResla}
0\longrightarrow P_d\stackrel{\partial_d}{\longrightarrow} P_{d-1}\stackrel{\partial_{d-1}}{\longrightarrow}\dots\stackrel{\partial_1}{\longrightarrow}P_0 \longrightarrow L^\la\longrightarrow0, 
\end{equation}
where the modules $P_m$ and the maps $\partial_m$ are defined as follows. 
For a subset $T\subseteq [1,d]$ and $t\in T$ define
$$
P_T:=q^{s(T)}E(\bi^{\linebr_T\la}),
$$
and the map 
\begin{equation}\label{EDeTt}
\partial_{T,t}:P_T\to P_{T\setminus\{t\}}
\end{equation}
to be the right multiplication by 
$$
\si(T,t)\cdot\left\{
\begin{array}{ll}
\psi(\linebr_T\la,\linebr_{T\setminus\{t\}}\la) &\hbox{if $t<d$}\\
z_\rho &\hbox{if $t=d$}
\end{array}
\right.
$$
In other words, 
\begin{align*}
\partial_{T,t}: x1_{\bi^{\linebr_T\la}}\mapsto\, &\si(T,t)x1_{\bi^{\linebr_T\la}}\psi(\linebr_T\la,\linebr_{T\setminus\{t\}}\la)\\=&\si(T,t)x\psi(\linebr_T\la,\linebr_{T\setminus\{t\}}\la)1_{\bi^{\linebr_{T\setminus\{t\}}\la}},
\end{align*}
if $t<d$, and
$$
\partial_{T,d}: x1_{\bi^{\linebr_T\la}}\mapsto \si(T,d)x1_{\bi^{\linebr_T\la}}z_\rho=\si(T,d)xz_\rho 1_{\bi^{\linebr_{T}\la}}.
$$
Note that $\partial_{T,t}$ is a homogeneous map. Now we have 
$$
P_m:=\bigoplus_{T\subseteq [1,d],\ |T|=m}P_T\qquad(0\leq m\leq d),
$$
and
$$
\partial_m=\bigoplus_{T\subseteq[1,d],\ |T|=m}\partial_T
$$
is a degree zero homomorphism where
$$
\partial_T=\sum_{t\in T}\partial_{T,t}:P_T\to P_{m-1}.
$$

Finally, the case of one row skew shape $\la=(d)$ is especially important, since in this case $L^{(d)}$ is the cuspidal module $L_\rho$, and we have the projective resolution $\Resol(L_\rho):=\Resol((d))$ of the cuspidal module. 

\subsection{Resolving proper standard and costandard modules}
Now consider a root partition $\pi=(\be_1\succeq\dots\succeq\be_l)\in\Pi(\al)$. 
We have the proper standard  
module
$
\bar \De(\pi):=q^{\shift(\pi)}\,L_{\be_1}\circ\dots\circ L_{\be_l}. 
$
By taking (outer) tensor product of resolutions $\Resol(L_{\be_1}),\dots,\Resol(L_{\be_l})$, we get the projective resolution $\Resol(L_{\be_1})\boxtimes\dots\boxtimes\Resol(L_{\be_l})$ of  
the $R_{\be_1,\dots,\be_l}$-module $L_{\be_1}\boxtimes\dots\boxtimes L_{\be_l}$. Since $\Ind_{\be_1,\dots,\be_l}^{\al}$ is exact and sends projectives to projectives, inducing to $R_\al$ and shifting grading by $\shift(\pi)$, yields a projective resolution 
$$\Resol(\bar\De(\pi)):=q^{\shift(\pi)} \Resol(L_{\be_1})\circ\dots\circ\Resol(L_{\be_l})$$ of $\bar \De(\pi)$. If $d=\height(\al)$, this resolution has length $\leq d$:  
\begin{equation}\label{EResDe}
0\longrightarrow P_d \stackrel{\partial_d}{\longrightarrow} P_{d-1} \stackrel{\partial_{d-1}}{\longrightarrow}\dots\stackrel{\partial_1}{\longrightarrow}P_0 \longrightarrow \bar\De(\pi)\longrightarrow0. \end{equation}

Recall also proper costandard modules from (\ref{EBarNabla}). By  Lemma~\ref{LDualInd},  we have 
$$\bar\nabla(\pi)\cong q^{\shift'(\pi)}L(\be_l)\circ\dots\circ L(\be_1)$$
for $\shift'(\pi):=-\shift(\pi)+\sum_{1\leq r<s\leq l}(\beta_r,\beta_s).
$ 
As for $\bar\De(\pi)$, we now get a projective resolution 
$$\Resol(\bar\nabla(\pi)):=q^{\shift'(\pi)} \Resol(L_{\be_l})\circ\dots\circ\Resol(L_{\be_1})$$ 
of the proper costandard module $\bar\nabla(\pi)$ of length $\leq d$. 


The resolutions $\Resol(\bar\nabla(\pi))$ and $\Resol(\bar\De(\pi))$ can be used to compute the global dimension of $R_\al$. First of all, the presence of the polynomial subalgebra in $d$ variables allows us to bound global dimension of $R_\al$ below as follows:

\begin{Lemma} 
Let $\O=F$ and $\al\in Q_+$ with $\height(\al)=d$. Then $\gldim R_\al\geq d$.
\end{Lemma}
\begin{proof}
For the polynomial algebra with $n$ variables, we have $\gldim F[y_1,\dots,y_n]=n$ by the  Hilbert Theorem on Syzygies, see for example \cite[Theorem 8.37]{Rotman}. Note by the Basis Theorem~\ref{TBasis} that $R_\al$ is a free module over its polynomial subalgebra. So the result comes from the McConnell-Roos Theorem \cite[Theorem 8.40]{Rotman}. 
\end{proof}

\begin{Theorem} 
Let $\al\in Q_+$ with $\height(\al)=d$, $\O=F$ and $L$ be an irreducible $R_\al$-module. Then $\prdim L\leq d$. 
\end{Theorem}
\begin{proof}
We have $L=L(\pi)$ for some $\pi\in\Pi(\al)$. We prove the theorem by the upward induction on the convex order. To start the induction, note that if $\pi$ is minimal, then $L(\pi)=\bar\De(\pi)$ by  Theorem~\ref{THeadIrr}(iv), and so it has a projective resolution $\Resol(\bar\De(\pi))$ of length $d$. 

For the inductive step, let $\pi\in\Pi(\al)$ and assume that the theorem has been proved for all $\pi'<\pi$. We also have a projective resolution $\Resol(\bar\nabla(\pi))$ of length $d$. It follows that $\EXT^k_{R_\al}(\bar\nabla(\pi),M)=0$ for all $k>d$ and all $M\in\Mod{R_\al}$. By the inductive assumption, we also have $\EXT^k_{R_\al}(L(\pi'),M)=0$ for all $k>d$, $M\in\Mod{R_\al}$, and $\pi'<\pi$. 

By Theorem~\ref{THeadIrr}(iv), all composition factors of $\bar\nabla(\pi)/L(\pi)$ are of the form $L(\pi')$ for some $\pi'<\pi$. So, using long exact sequences in cohomology, we conclude that $\EXT^k(\bar\nabla(\pi)/L(\pi),M)=0$ for all $k>d$ and $M\in\Mod{R_\al}$. Now the long exact sequence corresponding to the short exact sequence
$$
0\longrightarrow L(\pi)\longrightarrow \bar\nabla(\pi)\longrightarrow \bar\nabla(\pi)/L(\pi)\longrightarrow0
$$
looks like
\begin{align*}
\dots&\longrightarrow \EXT^k_{R_\al}(\bar\nabla(\pi)/L(\pi),M)\longrightarrow
\EXT^k_{R_\al}(\bar\nabla(\pi),M)\longrightarrow\EXT^k_{R_\al}(L(\pi),M)
\\
&\longrightarrow
\EXT^{k+1}_{R_\al}(\bar\nabla(\pi)/L(\pi),M)\longrightarrow\dots.
\end{align*}
We can now conclude that $\EXT^k_{R_\al}(L(\pi),M)=0$ for all $k>d$ and $M\in\Mod{R_\al}$. Theorefore $\prdim L(\pi)\leq d$ by \cite[Proposition 8.6]{Rotman}. 
\end{proof}

A general argument as explained for example in \cite{McN} now yields:

\begin{Corollary} 
Let $\al\in Q_+$ with $\height(\al)=d$. Then $\gldim R_\al=d$.
\end{Corollary}

A more explicit analysis of the terms $P_k$ of the resolution (\ref{EResDe}) implies:

\begin{Proposition} \label{TDeVan}
If $\EXT^k(\bar\De(\pi),V)\neq 0$ for some finitely generated $R_\al$-module $V$ and $k>0$, then $V$ has a composition factor $\simeq L(\si)$ for $\pi\leq \si$. 
\end{Proposition}

\begin{Corollary} \label{TDeNa}
Let $\pi,\si\in\Pi(\al)$. If $\pi\neq\si$, then $\EXT^k_{R_\al}(\bar\De(\pi),\bar\nabla(\si))=0$ for all $k\geq 0$.
\end{Corollary}
\begin{proof}
If $\pi\not\leq\si$, then, the Theorem holds by Proposition~\ref{TDeVan}.
If $\pi<\si$, then 
$$
\EXT^k_{R_\al}(\bar\De(\pi),\bar\nabla(\si))\cong
\EXT^k_{R_\al}(\bar\nabla(\si)^\circledast,\bar\De(\pi)^\circledast)\cong
\EXT^k_{R_\al}(\bar\De(\si),\bar\nabla(\pi))=0
$$
by Proposition~\ref{TDeVan} again. 
\end{proof}

\end{document}